\documentclass[11pt, letterpaper]{amsart}
\usepackage{amssymb, amsmath, amsthm}
\usepackage{hyperref}
\usepackage{stmaryrd}
\usepackage{bm}

\usepackage{tikz-cd}
\usetikzlibrary{positioning}

\newtheorem{theorem}{Theorem}[section]
\newtheorem{lemma}[theorem]{Lemma}
\newtheorem{proposition}[theorem]{Proposition}
\newtheorem{corollary}[theorem]{Corollary}

\theoremstyle{remark}
\newtheorem{remark}[theorem]{Remark}
\newtheorem{example}[theorem]{Example}
\theoremstyle{definition}
\newtheorem{definition}[theorem]{Definition}

\newcommand{\torus}{\mathbb{T}}
\newcommand{\CC}{\mathbb{C}}
\newcommand{\PP}{\mathbb{P}}

\newcommand{\ZZ}{\mathbb{Z}}
\newcommand{\QQ}{\mathbb{Q}}
\newcommand{\cO}{\mathcal{O}}

\DeclareMathOperator{\GL}{GL}
\DeclareMathOperator{\tot}{tot}
\DeclareMathOperator{\tr}{tr}
\DeclareMathOperator{\Hom}{Hom}
\DeclareMathOperator{\Gr}{Gr}

\newcommand{\f}{\mathbf{f}}
\newcommand{\g}{\mathbf{g}}

\newcommand{\op}[1]{\operatorname{#1}}

\usepackage{todonotes}
\definecolor{pistachio}{rgb}{0.58, 0.77, 0.45}
\definecolor{red(munsell)}{rgb}{0.95, 0.0, 0.24}
\definecolor{eggshell}{rgb}{0.94, 0.92, 0.84}
\definecolor{pink}{rgb}{0.858, 0.188, 0.478}

\title[Seiberg-like duality for resolutions of determinantal varieties]{Seiberg-like duality for resolutions of determinantal varieties}
\author{Nathan Priddis}
\address{Department of Mathematics, 275 TMCB, Brigham Young University, Provo, UT 84602, USA}
\email{priddis@math.byu.edu}

\author{Mark Shoemaker}
\address{Colorado State University, Department of Mathematics, 1874 Campus Delivery Fort Collins, Co, USA, 80523-1874}
\email{mark.shoemaker@colostate.edu}

\author{Yaoxiong Wen}
\address{School of Mathematics, Korea Institute for Advanced Study, Seoul 02455, Korea}
\email{y.x.wen.math@gmail.com}
\date{}

\begin{document}

\begin{abstract}
We study the genus-zero Gromov--Witten theory of two natural resolutions of determinantal varieties, termed the PAX and PAXY models. We realize each resolution as lying in a quiver bundle, and show that the respective quiver bundles are related by a quiver mutation.  We prove that generating functions of genus-zero Gromov--Witten invariants for the two resolutions are related by a specific cluster change of variables. Along the way, we obtain a quantum Thom--Porteous formula for determinantal varieties and prove a Seiberg-like duality statement for certain quiver bundles.
\end{abstract}

\maketitle
\tableofcontents
\clearpage

\section{Introduction}
In physics, one way that $\mathcal{N}=2$ superconformal field theories arise is via nonlinear sigma models with Calabi--Yau varieties as a target. 
From a mathematical perspective, this motivates one to compute the Gromov--Witten theory of Calabi--Yau varieties.

One of the most important tools we have for such computations is the quantum Lefschetz hyperplane theorem (see e.g. \cite{Giv, Bum, LLY1, LLY2, LLY3,  Lee}), which allows one to compute the genus zero Gromov--Witten theory of a hypersurface in terms of the theory of the ambient space.
This tool is especially effective when the Calabi--Yau in question is a complete intersection in a toric variety.  However, the majority of Calabi--Yau varieties are likely not realizable as complete intersections in toric varieties. In this paper we study determinantal varieties and their resolutions.  These provide a broad class of varieties which are not complete intersections.  
While it is as yet unknown how to compute the Gromov--Witten theory of general Calabi--Yau varieties, following \cite{JKL}, we obtain generating functions which allow one to compute genus zero Gromov--Witten invariants of (resolutions of) a general determinantal variety.

Roughly speaking, determinantal varieties are defined as follows: 
let $B$ be a smooth projective variety, and let $E$, $F$ be split vector bundles of rank $m$ and $n$, resp., over $B$, with $m \geq n$. Assume further that $E^\vee \otimes F$ is generated by sections. We can think of 
$E^\vee\otimes F$ as 
a family of  $n\times m$ matrices.  For a generic section $A \in \Gamma(B, E^\vee \otimes F)$, the matrix $A|_b$ at a generic point $b\in B$ is of full rank. Given a positive integer $s$, the determinantal variety $B(A,s)$ is defined as a locus in of $B$ where the rank of $A|_b$ is less than or equal to $s$. The study of determinantal varieties has a long history.  See \cite{Ber, GP1, GP2, GN, HM, HT, KK, Las, LR, Sch} among others. 

For $s \geq 1$, it is well known that the determinantal variety $B(A,s)$ is not a complete intersection.  Furthermore, it is singular in general, with a singular locus given by $B(A, s-1).$
There are two ways of constructing desingularizations of $B(A,s)$, known as the \emph{PAX} and \emph{PAXY} models respectively. 

The  PAX model resolution is based on the observation that $A_b$ has rank $s$ if and only if the kernel of $A_b$ is a subspace of $E_b$ of dimension $m-s$. Informally, the PAX desingularization of $B(A, s)$ parametrizes all subspaces of $\ker A_b$ of dimension $m-s$ as $b$ varies over $B$.  We denote this variety by $Z_A$. This space embeds naturally in a Grassmannian bundle on $B$. 

The second desingularization of $B(A, s)$---the \emph{PAXY} model---is based on the fact that $A_b$ has rank at most $s$ if and only if the matrix $A_b$ factors into a product of an $n\times s$ matrix $Y_b$ with an $s \times m$ matrix $X_b$. Informally, the PAXY desingularization parametrizes such factorizations of $A_b$ up to the natural action of $\GL_s.$

The PAX and PAXY models were studied in the physics paper \cite{JKL}, in which the authors argue that the two resolutions are related by a two dimensional version of Seiberg duality. The goal of the current work is to interpret this duality in the context of Gromov--Witten theory.  We compute $I$-functions for the two models and prove that they lie on the respective Lagrangian cones.  We compare  the two $I$-functions and show that they are related by an explicit change of variables.

\subsection{The connection to quiver bundles}\label{ss:qb}
An important step in the argument is to relate the PAX and PAXY resolutions to quiver varieties.  We prove the following.

\begin{theorem}[Theorem \ref{thm.pax/paxy_mutation}] \label{thm.intro_mutation}
The PAX and PAXY models can each be constructed as critical loci of specified superpotentials on fiber bundles over $B$ with fibers given by quiver varieties.  The respective quiver bundles are related by a mutation.     
\end{theorem}
The key tool in the theorem above is to generalize the construction of quiver varieties to the relative setting, essentially by replacing the vector spaces in the construction by vector bundles. 
We preview the precise formulation below.  
For a brief introduction to quiver varieties, see Section~\ref{ss.quivers} below. 

To construct the PAX model, define a quiver bundle, i.e. a family of quiver varieties, over $B$ via the following quiver:  
	
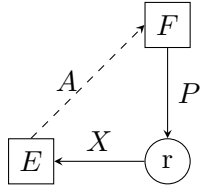
\begin{equation*}
		\begin{tikzpicture}
			\node[draw,
			circle,
			minimum size=0.6cm,
			] (gauge) at (0,0){r};
			
			\node [draw,
			minimum width=0.6cm,
			minimum height=0.6cm,
			left=1.2cm of gauge
			]  (frame-E) {$E$};
			
			\node [draw,
			minimum width=0.6cm,
			minimum height=0.6cm,
			above=1.2cm of gauge
			]  (frame-F) {$F$};
			
			\draw[-stealth] (gauge.west) -- (frame-E.east)
			node[midway,above]{$X$};
			
			\draw[-stealth] (frame-F.south) -| (gauge.north)
			node[near end,right]{$P$};
			
			\draw[-stealth,dashed] (frame-E.north) -- (frame-F.west)
			node[midway,left]{$A$};
		\end{tikzpicture}
	\end{equation*}
where $E$ and $F$ are vector bundles over the base $B$, the section $A\in \mathcal{H}om( E,  F)$ is a fixed map as in teh previous section, and $r = \op{rank}(E) - s$.  The  vector bundle associated to this quiver is $V = \mathcal{H}om(\cO_B \otimes \CC^r, E) \oplus \mathcal{H}om(F, \cO_B \otimes \CC^r)$, and the associated quiver bundle is the GIT quotient $V \sslash \GL_r.$  In the diagram above, $X$ and $P$ represent fiberwise coordinates on $V$.
The diagram also includes a dashed arrow, which is not standard. We use this notation to represent a fixed section $A\in \mathcal{H}om( E,  F)$, with the dashed line indicating that we are \emph{not} varying this section in moduli.

The cycle in the quiver above determines a superpotential 
\begin{equation*}
W_{PAX} = \tr (PAX): V \to \CC.
\end{equation*}
This function is gauge invariant under the $\GL_r$ action and therefore descends to a function on the quiver bundle.
The critical locus of this function in the GIT quotient $V \sslash \GL_r$, 
\[
Z_A := \mathrm{Crit}(W_{PAX}),
\]  
defines the PAX resolution.

Let us now perform a mutation of the above quiver at the gauged node.  After relabeling the fiber coordinates, we obtain the following quiver:
	
	
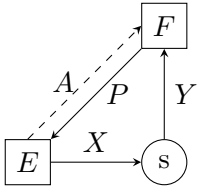
\begin{equation*}
		\begin{tikzpicture}
			\node [draw,
			circle,
			minimum size=0.6cm,
			] (gauge) at (0,0){s};
			
			\node [draw,
			minimum width=0.6cm,
			minimum height=0.6cm,
			left=1.2cm of gauge
			]  (frame-E) {$E$};
			
			\node [draw,
			minimum width=0.6cm,
			minimum height=0.6cm,
			above=1.2cm of gauge
			]  (frame-F) {$F$};
			
			\draw[-stealth] (frame-E.east) -- (gauge.west)
			node[midway,above]{$X$};
			
			\draw[-stealth] (gauge.north) -| (frame-F.south)
			node[near end,right]{$Y$};
			
			\draw[-stealth] (frame-F.south west) -- (frame-E.north east)
			node[midway,right]{$P$};
			
			\draw[-stealth,dashed] (frame-E.north) -- (frame-F.west)
			node[midway,left]{$A$};
		\end{tikzpicture}
	\end{equation*}
The vector bundle associated to this quiver is $V' = \mathcal{H}om(\cO_B \otimes \CC^s, F) \oplus \mathcal{H}om(E, \cO_B \otimes \CC^s) \oplus \mathcal{H}om(F, E)$ and the associated quiver bundle is $V'\sslash_{\theta'} \GL_s$. Again the dotted arrow represents a fixed section $A\in \mathcal{H}om( E,  F)$.  Here we use the superpotential
	\begin{align*}
		W_{PAXY}= \tr(P(A-YX)): V' \to \CC,
	\end{align*}
which is gauge invariant under the action of $\GL_s$. 
The critical locus of the function in the GIT quotient $V'\sslash_{\theta'} \GL_s$,  
\[
\widehat{Z}_A:=\mathrm{Crit}(W_{PAXY}),
\] 
defines the PAXY model.

Following \cite{JKL}, we prove that in fact the PAX and PAXY models are isomorphic (see Proposition~\ref{p.PAXY_res}).  However, the two presentations of the resolutions as critical loci in different quiver bundles as described above yield two different $I$-functions as we explain below.

\subsection{Gromov--Witten theory of determinantal varieties}

To study the genus zero Gromov--Witten theory of the PAX and PAXY models, we define $I$-functions $I_{Z_A}(q, q_1, z)$ and $I_{\widehat Z_A}(q', q_1', z)$, via explicit modifications of the $J$-function $J_B(q, u, z)$ of the base $B$. These $I$-functions are obtained using the two quiver descriptions of the previous section.  The function $I_{Z_A}(q, q_1, z)$ arises from the PAX model, in which $Z_A$ is embedded in $\Gr_B(r, E).$  The function
 $I_{\widehat Z_A}(q', q_1', z)$ from the PAXY model embedding.

Although $Z_A$ and $\widehat Z_A$ are isomorphic, these two different embeddings yield \emph{distinct} $I$-functions.  From a physics perspective, this is because the two gauged linear sigma models (GLSMs) are different.  From a mathematical perspective, it is a consequence of the fact that
 relative quasimap invariants (see \cite{Oh, SupTse}) of the two GIT quotients have no direct connection with one another.

Using the Abelian/non-Abelian correspondence as formulated in \cite{CLS} together with the quantum Lefschetz theorem in \cite{CoatesQLHT} we prove in Proposition~\ref{p.I_PAX} that the PAX model $I$-function $I_{Z_A}(q, q_1, z)$ lies on the Lagrangian cone of $Z_A$.  We note that in the special case that the determinantal variety $B(A,s)$ is smooth, this yields an $I$-function for $B(A,s)$ itself.  The result may therefore be viewed as a generalization of the quantum Lefschetz theorem to the more general setting of determinantal varieties.  

We also verify that the PAXY model $I$-function $I_{\widehat Z_A}$ lies on the Lagrangian cone, but in this case the proof is more difficult.  Indeed the function $I_{\widehat Z_A}$ involves twisting by a non-concavex vector bundle, thus one cannot apply a quantum Lefschetz-type result as in the case of $I_{Z_A}$.  Interestingly, the result in this case becomes a corollary of duality between $I_{Z_A}$ and $I_{\widehat Z_A}$, which we describe in the next section.

\subsection{Seiberg-like duality for the PAX and PAXY resolution}

For quiver varieties related by mutation, there is a predicted duality relating the associated $I$-functions by an explicit change of variables.
This Seiberg-like duality was first explored in \cite{BPZ} and mathematically formulated in \cite{Ruan17} 
under the name \emph{mutation conjecture}. 
The conjecture was proven in \cite{Don, DW} for the simplest possible mutation---that between the Grassmannian and dual Grassmannian.  The result was generalized in \cite{Zha} and \cite{HZ} to the case of type $A_n$ quivers and type $D_4$ quivers, respectively.

In light of Theorem~\ref{thm.intro_mutation}, it is natural to expect that the $I$-functions $I_{Z_A}$ and $I_{\widehat{Z}_A}$ for the PAX and PAXY models are related to each other in an analogous way to that described in the mutation conjecture.  This is the main result of the paper.

\begin{theorem}[Theorem \ref{Thm.PAX/PAXY}] \label{thm.intro_main}

Let $B(A,s)$ be the determinantal variety, with PAX and PAXY resolutions $Z_A$ and $\widehat Z_A$.  
We have the following relations between the PAX and PAXY model $I$-functions. 
\begin{enumerate}
		\item When $m \geq n+2$, then
		\begin{align*}
			I_{\widehat{Z}_A}(q, q_1) = I_{Z_A}(q', q_1').
		\end{align*}
        under the change of variables 
    \begin{equation}\label{intro.paxpaxycov}
            q'_1=q_1\quad\text{ and }\qquad q'=q\cdot q_1^{\langle c_1(E), -\rangle}.
	 \end{equation}
		
		\item When $m = n+1$, then
		\begin{align*}
			I_{\widehat{Z}_A}(q, q_1) = e^{(-1)^r q_1/z} I_{Z_A}(q', q_1').
		\end{align*}
    under the change of variables \eqref{intro.paxpaxycov}.
		
		\item When $m = n$, then
		\begin{align*}
			I_{\widehat{Z}_A}(q, q_1) = (1+(-1)^{r} q_1)^{\Psi} I_{Z_A}(q', q_1').
		\end{align*}
    under the change of variables 
	\[
	q'_1=q_1\quad\text{ and }\qquad q'=\frac{q\cdot q_1^{\langle c_1(E), -\rangle}}{(1+(-1)^{r} q_1)^{ \langle c_1(E),-\rangle-\langle c_1(F), - \rangle}},
	\]
    where $\Psi=r+(c_1(F) - c_1(E))/z$. 
\end{enumerate}
\end{theorem}

The generalization from quiver varieties to quiver bundles described in Section~\ref{ss:qb} has implications beyond just the study of determinantal varieties, and suggests a natural generalization of Seiberg-like duality for quiver varieties to the case of families of quiver varieties over a base $B$.  

We explore this perspective in Section~\ref{s.generalized-Seiberg-duality} and prove such a duality for Grassmannian bundles and their duals in Theorem~\ref{Them.building_block}.  To illustrate the benefit of the relative perspective, we note that  Theorem~\ref{Them.building_block} already implies new cases of the mutation conjecture in the absolute case. See Remark~\ref{Rem:new case} below.

In Section~\ref{s.general_paxpaxy} we introduce determinantal varieties, the PAX and PAXY resolutions, and their realization as critical loci inside quiver bundles.  We then apply the methods from the basic mutation studied in Section~\ref{s.generalized-Seiberg-duality} to prove Theorem \ref{thm.intro_main}.

\section*{Acknowledgements}
We would like to thank Tom Coates, Sanghyeon Lee, Wendelin Lutz, Jeongseok Oh, Yongbin Ruan, Qaasim Shafi, Rachel Webb, and Yingchun Zhang  for the useful discussions. N. Priddis is supported by the Simons Foundation Travel Grant 586691.
M. Shoemaker is supported by the Simons Foundation Travel Grant 958189.
Y. Wen is supported by a KIAS Individual Grant (MG083902) at the Korea Institute for Advanced Study.

\section{Seiberg-like duality for Grassmannian bundles}\label{s.generalized-Seiberg-duality}

In this section we introduce quiver varieties and mutation of quivers, and describe a conjectural Seiberg-like duality relating the $I$-functions of mutated quivers.  We then generalize the discussion to the setting of quiver bundles over an arbitrary base $B$, and prove a duality statement for Grassmannian bundles and their duals.
The methods of this section will be used again in Section~\ref{s.general_paxpaxy} to relate the $I$-functions of the PAX and PAXY resolutions.

\subsection{Quivers and mutations}\label{ss.quivers}
A quiver is a directed graph $Q = (Q_{0}, Q_{1})$ with vertices $Q_{0}$ and edges $Q_{1}$.  The vertices are partitioned into \emph{framed} and \emph{gauge} vertices: $Q_{0} = Q_{\f} \coprod Q_{\g}$.  Given a quiver, we construct a GIT quotient as follows.  For each vertex $i \in Q_{0}$, assign a finite dimensional complex vector space $V_i$.  If $i \in Q_{\g}$ is a gauge vertex, then choose a reductive group $G_i \subseteq \GL(V_i)$.  Define
\[
V = \bigoplus_{e: i \to j \in Q_{1}} \Hom(V_i, V_j), \quad G = \prod_{i \in Q_{\g}} G_i.
\]
The group $G$ acts on $V$ by the balanced action, i.e. $(A,B)\cdot g=(Ag, g^{-1}B)$. For a given choice of character $\theta \in \Hom(G, \CC^*)$, we obtain a GIT quotient 
\[
Y = V \sslash_\theta G.\]
For every \emph{cycle} $c = i_0 \to i_1 \to \cdots \to i_k \to i_0$ in $Q$, there is an associated $G$-invariant function 
\[
f_c := \tr \left( A_{i_k \to i_0 } \circ A_{i_{k-1} \to i_k} \circ \cdots \circ A_{i_0 \to i_1} \right): V \to \CC
\]
where $A_{i_{j-1} \to i_j  } \in \Hom(V_{i_{j-1}}, V_{i_j})$. A function $\phi: \{\text{cycles $c$ in $Q$}\} \to \CC$ therefore defines a $G$-invariant function
\[
\sum_c \phi(c) \cdot f_c: V \to \CC,
\]
which descends to a function (called a \emph{superpotential}) $w: Y \to \CC$ on the GIT quotient. We obtain a Landau--Ginzburg model $(Y, w)$, which we call the quiver LG model.  If $Q$ has no cycles, then there is no superpotential, and $Y$ is a quiver variety. We will assume for simplicity that the critical locus 
\[
\op{Crit}(\phi): Y \to \CC
\]
is smooth, and we consider the Gromov--Witten theory of $\op{Crit}(\phi)$.  In the more general setting where $\op{Crit}(\phi)$ is not smooth, one could instead consider the GLSM theory of the LG model $(Y, w)$ (see e.g. \cite{FJR2, CFGKS, TX2, Shoe, FKim, CZ}).
	
Fix a quiver $Q$ with an assignment of vector spaces $V_i$ for each vertex as above.  Assume the associated gauge groups are $G_i = \GL(V_i)$ for each $i \in Q_{\g}$, and let $(Y,w)$ denote the quiver LG model.  Fix a gauge vertex $k \in Q_{\g}$.  The \emph{mutation} of $Q$ at gauge node $k$ is a new quiver $Q'$ with associated vector spaces constructed as follows:
\begin{enumerate}
	\item for each sequence of two arrows $i \to k \to j$ passing through $k$, add a new arrow $i \to j$, and reverse the direction of the arrows $i \to k$ and $k \to j$;
	\item set $N = \op{max}(N_{in}, N_{out}) - \dim(V_k)$ where $N_{in}$ (resp. $N_{out}$) is the sum $\sum_{i \in Q_{0}} n_i \dim(V_i)$ and $n_i$ is the number of incoming edges $i \to k$ (resp. outgoing edges $k \to i$). Replace the vector space $V_k$ with $\CC^N$, and $G_k = \GL(V_k)$ with $\GL_N$;
 \item if any 2-cycles arise in the first step, delete both edges of each 2-cycle;
	\item for each new cycle $c$ of length greater than 2 arising in the first step, add a generic multiple of $f_c$ to the superpotential.  
\end{enumerate}

This procedure yields a new LG model $(Y', w')$. It was observed in \cite{BPZ} that the two quivers and their associated LG models are related by a two-dimensional $\mathcal{N}=(2,2)$ version of Seiberg-like duality.  Mathematically, this suggests that the  generating functions of Gromov--Witten invariants of the respective critical loci (or, more generally, of GLSM invariants of the LG models) are related by a mutation of Novikov parameters.  This was described in \cite{Ruan17} as the mutation conjecture. In \cite{Don,DW}, Dong and the third named author  proved the mutation conjecture (for small $I$-functions) in genus zero for the Grassmannian and its dual Grassmannian.  This result has served as the ``basic building block'' for proving other instances of the duality. By leveraging this building block, Zhang in \cite{Zha} proved the conjecture for quivers of type $A_n$, and He-Zhang in \cite{HZ} proved the conjecture for star-shaped and type $D_4$ quivers. 
In this section, we generalize the methods of Zhang given in \cite{Zha} to prove a similar statement in the relative setting.
Such quiver bundles have also appeared in the study of the moduli space of Higgs bundles in \cite{AG03}.

\subsection{The basic relative quiver mutation} \label{s.twisted_Gr_bdl}
In this paper we propose generalizing the Seiberg-like duality described above to the relative setting by working over a general base $B$ and replacing vector spaces at framed nodes by vector bundles.
We will focus on the following special case.

 Let $B$ be a smooth projective variety, and let $E=\bigoplus_{i=1}^m L_i$ and $F=\bigoplus_{i=1}^n M_i$ be vector bundles on $B$ of rank $m$ and $n$, respectively. We will assume that $ m \geq n $ and that $E^{\vee}\otimes F$ is generated by global sections. 

Consider the following quiver bundle:
\begin{equation}\label{quiver1}
\begin{tikzpicture}
\node[draw,
    circle,
    minimum size=0.6cm,
] (gauge) at (0,0){r};

\node [draw,
    minimum width=0.6cm,
    minimum height=0.6cm,
    left=1.2cm of gauge
]  (frame-E) {$E$};

\node [draw,
    minimum width=0.6cm,
    minimum height=0.6cm,
    above=1.2cm of gauge
]  (frame-F) {$F$};

\draw[-stealth] (gauge.west) -- (frame-E.east)
    node[midway,above]{$X$};

\draw[-stealth] (frame-F.south) -| (gauge.north)
    node[near end,right]{$P$};
\end{tikzpicture}.
\end{equation}

In the notation of the previous section, the framed vertices (in $Q_{\f}$) are depicted by square nodes. There is a single gauge vertex (in $Q_{\g}$) depicted by a round node. The set of edges $Q_{1}$ contains the two arrows depicted in the diagram. However, different from the description in the previous section, instead of assigning vector spaces to each element in $Q_{\f}$, we assign the vector bundles $E$ and $F$. 

From this data, one can construct a GIT quotient, which will be a fiber bundle over the base $B$ whose fibers are quiver varieties. First, define the vector bundle 
\[
V = \mathcal{H}om( \cO_B \otimes \CC^r, E) \oplus \mathcal{H}om(F, \cO_B \otimes \CC^r).
\]
		
The action of $\GL_r$ on $\CC^r$ naturally induces an action on $V$.  On an open subset $U \subset B$ where $\cO_B$, $E$ and $F$ are trivialized, let $X$ and $P$ denote coordinates on $\mathcal{H}om( \cO_B \otimes \CC^r, E)$ and $\mathcal{H}om(F, \cO_B \otimes \CC^r)$, respectively.  Then the
$\GL_r$ action is locally given by
\[
(X,P)\cdot g := (X\cdot g, g^{-1}\cdot P).
\]
After choosing the character ${\theta_+}(g):=\det(g)$, we obtain the GIT quotient 
\begin{align*}
	Z:=V \sslash_{\theta_+} \GL_r = \op{tot} \left( \begin{tikzcd}
		S_r \otimes \pi_r^* F^\vee  \ar[d]  \\ 
		\Gr_B(r,E) 
	\end{tikzcd} \right). 
\end{align*}
Here $\Gr_B(r,E)$ denotes the Grassmannian bundle over the base $B$ of $r$-planes in the fibers of $E$. We will suppress the subscript $B$ from now on and simply denote this as $\Gr(r,E)$. The bundle $S_r$ denotes the tautological bundle on $\Gr(r,E)$, and the map $\pi_r:\Gr(r,E)\to B$ is the projection. 

There is a single mutation possible on the quiver in \eqref{quiver1} at the single gauged node.
After applying the mutation procedure, we obtain the following quiver bundle
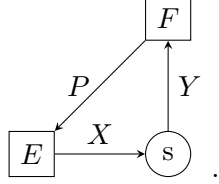
\begin{equation*}
\begin{tikzpicture}
\node[draw,
    circle,
    minimum size=0.6cm,
] (gauge) at (0,0){s};

\node [draw,
    minimum width=0.6cm,
    minimum height=0.6cm,
    left=1.2cm of gauge
]  (frame-E) {$E$};

\node [draw,
    minimum width=0.6cm,
    minimum height=0.6cm,
    above=1.2cm of gauge
]  (frame-F) {$F$};

\draw[-stealth] (frame-E.east) -- (gauge.west)
    node[midway,above]{$X$};

\draw[-stealth] (gauge.north) -| (frame-F.south)
    node[near end,right]{$Y$};

\draw[-stealth] (frame-F.south west) -- (frame-E.north east)
    node[midway,left]{$P$};
\end{tikzpicture}.
\end{equation*}
where $s=m-r$.

To form the GIT quotient for this quiver,  take \[V'=( E\otimes F^\vee)\oplus \mathcal{H}om( \cO_B \otimes \CC^s, F) \oplus \mathcal{H}om(E, \cO_B \otimes \CC^s),
\] 
with an action of $\GL_s$ induced by the  action on $\CC^s$.  On an open subset where $E $ and $F $ are trivialized, let $P$, $Y$, and $X$ denote coordinates on each of the respective summands of $V'$.
In these coordinates, the action of $\GL_s$ is given by
\[
(P,Y,X)\cdot g= (P, Y\cdot g, g^{-1}\cdot X).
\]
Denote by $\pi_s: \Gr(s, E ^\vee) \rightarrow B$ the projection. If we choose the negative character ${\theta_-}(g)=(\det(g))^{-1}$, we obtain the GIT quotient
\begin{align*}
	Z^{\op{tot}}:= V' \sslash_{\theta_-} \GL_s = \op{tot} \left( \begin{tikzcd}
		( S_s \otimes \pi_s^* F)  \oplus \pi_s^* ( E\otimes F^\vee)  \ar[d]  \\ 
		\Gr(s, E^\vee) 
	\end{tikzcd} \right). 
\end{align*}  

Following the construction in Subsection \ref{ss.quivers}, since there is a loop in the quiver, the GIT quotient $Z^{\op{tot}}$ comes with a nontrivial superpotential, given by
\begin{align*}
	W=\tr(PYX).
\end{align*}
Define $\widehat{Z}$ to be the critical locus of $W$:
\[
\widehat{Z} := \operatorname{Crit}(W).
\]

In this case, the critical locus of $W$ is defined by the equations $PY=0$, $YX=0$, and $XP=0$. Since $X$ is full rank, it must be the case that $Y=0$, so we have 
\[
\widehat{Z}=\{XP=0\}\subset \op{tot}( \pi_s^*(E\otimes F^\vee )).
\]
	
Thus $\widehat{Z}$ is the zero locus of the regular section $XP$ of the vector bundle $\pi_{\op{tot}}^*( S_s^\vee \otimes F^\vee )$, where $\pi_{\op{tot}}: \op{tot}( \pi_s^* ( E\otimes F^\vee ) ) \rightarrow \Gr(s, E^\vee)$ is the projection. 

\begin{lemma}\label{lem:iso_basic_mutation}
    Under Grassmannian duality, we have a canonical isomorphism $Z \cong \widehat{Z}$.
\end{lemma}
\begin{proof}
    Over $\pi_{\op{tot}}: \op{tot}( \pi_s^* ( E\otimes F^\vee ) ) \rightarrow \Gr(s, E^\vee)$, we have the tautological exact sequence
    \[
    0 \rightarrow \pi_{\tot}^* S_s \rightarrow \pi_{\tot}^*\pi_s^* E^\vee \rightarrow \pi_{\tot}^* Q_s \rightarrow 0.
    \]
    Here, we use the index $s$ in $Q_s$ to emphasize the dependence of the Grassmannian $\Gr(s, E^\vee)$. 
    
    Dualizing and tensoring with $\pi_{\tot}^*\pi_s^* F^\vee$, we have
    \[
    0 \rightarrow \pi_{\tot}^*(Q_s^\vee \otimes \pi_s^*F^\vee) \rightarrow \pi_{\tot}^*\pi_s^*(E \otimes F^\vee) \rightarrow \pi_{\tot}^*(S_s^\vee \otimes \pi_s^*F^\vee) \rightarrow 0.
    \]
    Notice that $XP$ is the composition of the canonical section of $\pi_{\tot}^*\pi_s^*(E \otimes F^\vee)$ with the second to last arrow above. By Grassmannian duality, we have a canonical isomorphism $\Gr(s, E^\vee) \cong \Gr(r, E)$ and under this isomorphism the vector bundles $Q_s^\vee$ and $S_r$ are identified.  The conclusion follows.
\end{proof}

\subsection{The quasimap $I$-functions}\label{ss:Ifunctions}
In this section we compute $I$-functions for $Z$ and $\widehat{Z}$ and prove that they lie on the respective Lagrangian cones.  That is, they determine the genus-zero Gromov--Witten theory of $Z$ and $\widehat{Z}$.  More conceptually, these generating functions will be defined in terms of twisted relative quasimap invariants in the sense of \cite{Oh}.

We will work with equivariant $I$-functions, with respect to the following torus action. Let $(\CC^*)^m$ act on $E$ and $(\CC^*)^n$ act on $F$ fiberwise. This induces an action of the torus $\torus:= (\CC^*)^m \times (\CC^*)^n$ on both $Z$ and $\widehat{Z}$. For the latter space, note that the section $XP \in \Gamma( \op{tot}( \pi_s^* (E\otimes F^\vee ) ), \pi_{\op{tot}}^* (S_s^\vee\otimes \pi_s^* F^\vee))$ is $\torus$-equivariant.

Let $x_i=c_1^{\torus}(L^\vee_i)$ and $\hat x_i:=\pi_r^*(x_i)$, let $z_j=c^{\torus}_1(M_j^\vee)$ and $\hat z_j=\pi_r^*(z_j)$, and let $\{ y_i \}_{i=1}^r$ be the Chern roots of $S^\vee_r$. 

Let $J_B(q, u, z)$ denote the $J$-function of the base $B$, where $u$ is a coordinate of $H^*(B)$.
Let $J_\beta = J_\beta(u, z)$ denote the coefficient of $q^\beta$ in $J_B$, namely:
\[
J_B(q, u, z) =\sum_\beta q^\beta J_\beta.
\]
The $I$-functions for both $Z$ and $\widehat{Z}$ will be obtained by hypergeometric modifications of the coefficients of $J_B$. For notational simplicity, we will often suppress the $u$ variable. 

For $Z$, we define $I_{Z} =  I_{Z}(q, q_1, z)$ to be the following modification of $J_B$:
\begin{align} \label{IX}
	I_{Z}(q, q_1, z) = I_{\Gr(r,E)}^{ e_{\torus}^{-1}(S_r \otimes \pi_r^* F^\vee)}= \sum_\beta q^\beta \sum_d q_1^d\sum _{\substack{(d_1,\dots,d_r)\\ \sum d_i=d}}N_{\beta,\vec{d}} \ J_\beta, 
\end{align}
where
\begin{align*}
	N_{\beta,\vec{d}}=&\prod_{\substack{i,j=1\\ i\neq j}}^r 
	\frac{\prod_{h\leq d_j-d_i}(y_j-y_i+hz)}{\prod_{h\leq 0}(y_j-y_i+hz)} \cdot \prod_{j=1}^r \prod_{i=1}^{m} \frac{\prod_{h\leq 0} (y_j-\hat x_i+hz)}{\prod_{h\leq d_j-\beta\cdot x_i} (y_j-\hat x_i+hz)}\\
	\cdot & \prod_{j=1}^r\prod_{k=1}^{n}\frac{\prod_{h\leq 0}(-y_j+\hat z_k+hz)}{\prod_{h\leq -d_j+\beta\cdot z_k}(-y_j+\hat z_k+hz)}.
\end{align*}
Here the expression $\prod_{h \leq b}( - )$ is used to denote the product $\prod_{h=- \infty}^b( - )$, and  we denote
\[
    \beta \cdot x_i := \langle c_1(L_i^\vee), \beta \rangle, \quad \beta \cdot z_k := \langle c_1(M^\vee_k), \beta\rangle.
\]

\begin{remark}\label{finitesum1}
    We note that for a fixed $\beta \in H_2(B)$ 
    the product 
\[
\prod_{i=1}^{m} \frac{\prod_{h\leq 0} (y_j-\hat x_i+hz)}{\prod_{h\leq d_j-\beta\cdot x_i} (y_j-\hat x_i+hz)}
\]
vanishes whenever $d_j< \min_i (\beta\cdot x_i)$ for some $j \leq r$.
This implies that: 
    \begin{enumerate}
        \item for a fixed $\beta \in H_2(B)$, the integers $d \in \ZZ$ which contribute to the second summation is bounded below; and
        \item for a fixed $d \in \ZZ$, there are only finitely many $(d_1, \ldots, d_r) \in \ZZ^r$ such that $\sum d_i = d$ for which $N_{\beta,\vec{d}}$ is nonzero.
    \end{enumerate}
\end{remark}

We have an identification of  $\QQ[H_2(\Gr(r, E))]$ with $\QQ[H_2(B)][q_1, q_1^{-1}]$ 
given by
\begin{equation}\label{Novmap} q^{\tilde \beta} \mapsto q^{(\pi_r)_*(\tilde \beta)} q_1^{\langle c_1(S_r^\vee), \tilde \beta \rangle}.
\end{equation}
Using the map \eqref{Novmap} and Remark~\ref{finitesum1},
we may identify the Novikov ring of $\Gr(r, E)$, denoted 
$\QQ[\![\op{Eff}(\Gr(r, E))]\!]$, with a subring of 
$\QQ[\![\op{Eff}(B)]\!](\!(q_1)\!).$
Under this identification,
$I_{Z}$ lies in the $\torus$-equivariant Givental space $H^*(Z)(\!(z^{-1})\!)$.

\begin{proposition} \label{p.I-function-Gr}
 $I_{Z}$ lies in the Givental's ($\torus$-equivariant) Lagrangian cone of $Z$.
\end{proposition}
\begin{proof}
	
Let $\bm{t} = t c_1^{\torus}(S_r^\vee) \in H^2_{\torus}(\Gr(r, E))$.
Define $I_{Z}(q, q_1, t, z)$ to be 
\[
I_{Z}(q, q_1, t,z):= e^{\bm{t}/z}I_{Z}(q, q_1e^t, z).
\]

We will show that  $I_{Z}(q, q_1, t, z)$ lies on the Lagrangian cone.  Note that $I_{Z}(q, q_1, t, z)$ is a hypergeometric modification of the $I$-function of ${\Gr(r, E)}$ obtained by Oh in \cite{Oh}. In fact, as described in \cite{CLS}, it is the \emph{Givental--Martin} modification of a twisted $I$-function on an associated Abelian GIT quotient, so the proposition essentially follows from \cite[Theorem 1.11]{CLS}.

More precisely, the proposition follows by applying the argument of Theorem~5.11 of \cite{CLS}, which provides an Abelian/non-Abelian correspondence for twisted theories over a flag bundle.  Because that theorem is not stated in exactly the generality that we need here, we summarize the argument below and explain the necessary adjustments to the case at hand.

Let $T$ be a maximal torus in $\GL_r$, and let $\Gr_T$ denote the GIT quotient $\mathcal{H}om(\CC^r, E) \sslash T$. This is a toric bundle over $B$. (The reader should take care not to confuse $T$ with the torus $\torus$, as defined earlier.) Let $I_{\Gr_T}$ denote the Brown $I$-function for $\Gr_T$ as in \cite{Brown}, and let 
$I_{\Gr_T}^{e^{-1}_{\torus}( (S_r \otimes \pi_r^* F^\vee)_T)}$ denote the corresponding 
$((S_r \otimes \pi_r^* F^\vee)_T, e^{-1}_{\torus}(-))$-twisted $I$-function, where $(S_r \otimes \pi_r^* F^\vee)_T \to \Gr_T$ is the vector bundle on the Abelian quotient $\Gr_T$ corresponding to $(S_r \otimes \pi_r^* F^\vee) \to \Gr(r, E)$.  

Following the notation and discussion in Section~5.2 of \cite{CLS} (in particular Theorem~5.11), we apply a further twist to $I_{\Gr_T}^{e^{-1}_{\torus}( (S_r \otimes \pi_r^* F^\vee)_T)}$ by $(\Phi, \mathbf c')$ where $\Phi \to \Gr_T$ is the root bundle with respect to $\GL_r$ and $\mathbf c'$ is the $\CC^*$-equivariant Euler class with respect to the $\CC^*$-action scaling the fibers of $\Phi$.  Denote the equivariant parameter of this $\CC^*$ action by $\lambda$. Projecting via the map on Givental spaces defined in \cite{CLS}
\[
p: \mathcal{H}_{\Gr_T}^W \to \mathcal{H}_{\Gr(r, E)}
\]
and taking the nonequivariant limit $\lambda \mapsto 0$, we obtain $I_{Z}(q, q_1, t, z)$.  

By the proof of \cite[Theorem~5.11]{CLS}, this function lies in the $((S_r \otimes \pi_r^* F^\vee), e^{-1}_{\torus})$-twisted Lagrangian cone of $\Gr(r, E)$.  While the theorem in \emph{loc. cit.} is only stated for the twisted theory $( L, e_{\CC^*}(-))$, where $L$ is a line bundle and $e_{\CC^*}(-)$ is an equivariant Euler class, the result holds for other multiplicative classes as well by \cite[Remark 3.7]{CLS}. The further extension to the case that $L$ is replaced by the sum of line bundles also follows the same argument.

By definition, the $((S_r \otimes \pi_r^*  F^\vee), e^{-1}_{\torus})$-twisted Lagrangian cone $\Gr(r, E)$ is equal to the $\torus$-equivariant Lagrangian cone of $Z$.

\end{proof}

Next, we compute an $I$-function for $\widehat{Z}$. By abuse of notation, let $\hat x_i:=\pi_s^*(x_i)$, $\hat z_j=\pi_s^*(z_j)$, and let $w_1,\dots, w_s$ be the Chern roots of $S_s^\vee$, then we define 
\begin{align}
	I_{\widehat{Z}}(q ', q_1', z) = I_{\Gr(s, E^\vee)}^{e_{\torus}( S_s^\vee \otimes \pi_s^* F^\vee)e^{-1}_{\torus}(\pi_s^*(E\otimes F^\vee))} =\sum_\beta {q '}^\beta \sum_d q_1'^d\sum _{\substack{(d_1,\dots,d_s)\\ \sum d_i=d}}\widehat{N}_{\beta,\vec{d}} \cdot J_\beta,
\end{align}
where
\begin{align*}
	\widehat{N}_{\beta,\vec{d}}=
    &\prod_{\substack{i,j=1\\ i\neq j}}^s 
	\frac{\prod_{h\leq d_j-d_i}(w_j-w_i+hz)}{\prod_{h\leq 0}(w_j-w_i+hz)} \cdot \prod_{j=1}^s \prod_{i=1}^{m} \frac{\prod_{h\leq 0} (w_j+\hat x_i+hz)}{\prod_{h\leq d_j+\beta\cdot x_i} (w_j+\hat x_i+hz)}\\
	\cdot & \prod_{k=1}^{n}\prod_{j=1}^s\frac{\prod_{h\leq d_j+\beta\cdot z_k}(w_j+\hat z_k+hz)}{\prod_{h\leq 0}(w_j+\hat z_k+hz)}\cdot \prod_{i=1}^m\prod_{k=1}^{n} \frac{\prod_{h\leq 0}(-\hat x_i+\hat z_k+hz)}{\prod_{h\leq -\beta\cdot x_i+\beta\cdot z_k}(-\hat x_i+\hat z_k+hz)}.
\end{align*}
Analogously to the case of $I_Z$, the modification factor $\widehat{N}_{\beta,\vec{d}}$ vanishes whenever 
\[
d_j < \min_i(- \beta \cdot x_i)
\]
for some $j$.
We have the identification of 
$\QQ[H_2(\Gr(s, E^{\vee}))]$ 
with $\QQ[H_2(B)][q_1', (q_1 ')^{-1}]$ given by
\begin{equation}\label{Novmap2} 
{q}^{\tilde \beta} \mapsto ({q '})^{\pi_*(\tilde \beta)} ({q_1'})^{\langle c_1(S_s^\vee), \tilde \beta \rangle},
\end{equation}
which allows us to identify the Novikov ring $\QQ[\![ \op{Eff}(\Gr(s, E^\vee))]\!]$ with a subring of 
\[
\QQ[\![\op{Eff}(B)]\!](\!(q_1 ')\!).
\]

Under this identification,
$I_{\widehat{Z}}$ lies in the $\torus$-equivariant Givental space $H^*(\widehat{Z})(\!(z^{-1})\!)$.

Applying the same argument as in Proposition~\ref{p.I-function-Gr}, together with the quantum Lefschetz theorem as given in \cite{CoatesQLHT}, we have the following.

\begin{proposition}\label{p.I-function-Gr-dual} 
	$I_{\widehat{Z}}$ lies in the $\torus$-equivariant Lagrangian cone of $\widehat{Z}$.
\end{proposition} 

Hence, we have two $I$-functions $I_Z$ and $I_{\widehat{Z}}$, each lying on the $\torus$-equivariant Lagrangian cones of the respective quiver bundles. As $Z$ and $\widehat{Z}$ are related by mutation, the respective $I$-functions should be related by Seiberg-like duality. We prove this in the next section.

\subsection{The proof of basic Seiberg-like duality} \label{s.proof:Gr}
In this section we prove a correspondence between $I_Z$ and $I_{\widehat{Z}}$ stated in Theorem~\ref{Them.building_block} below.

By Lemma~\ref{lem:iso_basic_mutation}, the isomorphism $Z \cong \widehat{Z}$ induces an isomorphism of their cohomologies. Recall the  
torus actions on $E$ and $F$ induce a $\torus$-action on both $Z$ and $\widehat{Z}$. Under the Grassmannian duality, there is a natural bijection between the fixed loci.

To describe the fixed locus, we introduce the following notation. For a positive integer $k$, let $[k]=\{1,\dots,k\}$. For $r\leq m$, let $\mathbf{i}=\{i_1<i_2<\dots< i_r\}\subseteq [m]$ denote a subset of size $r$. Furthermore, let $\mathbf{K}_r=\{\mathbf{i}\subset [m]\}$ denote the set of all such subsets.

\begin{lemma}\label{lem:fix1}
	The fixed locus of the $\torus$-action on $Z$ is a disjoint union $\bigsqcup_{\mathbf{i}}F_{\mathbf{i}}$ indexed by the set $\mathbf{K}_r$. Let $\iota_{F_{\mathbf{i}}}:F_{\mathbf{i}}\hookrightarrow Z$ denote the corresponding component. Then $F_{\mathbf{i}}\cong B$ and $ S_r|_{F_{\mathbf{i}}}\cong \bigoplus_{k=1}^r  L_{i_k}$. 
\end{lemma}
\begin{proof}
Since the $\torus$-action on $Z=\tot(S_r\otimes\pi_r^*F^\vee \rightarrow \Gr(r,E))$ is induced from the 
torus actions on $E$ and $F$, the $\torus$-fixed loci are the fixed points in $\Gr(r,E)$ represented by the standard rank $r$ subbundles, namely those of the form $\oplus_{k=1}^rL_{i_k} \hookrightarrow E$. Thus each component of the fixed locus is isomorphic to $B$ and parametrized by $\mathbf{K}_r$.

\end{proof}

\begin{corollary} \label{cor:fix1}
With the same notation as before, namely $x_i=c_1^{\torus}(L^\vee_i)$ and $y_1,\dots, y_r$ are the equivariant Chern roots of $S_r^\vee$, we have $\iota_{F_{\mathbf{i}}}^*(y_k)=x_{i_k}$. 
\end{corollary}

For $\widehat{Z}$, we also have an action of $\torus$, and the fixed locus has a similar description, but this time as subsets of $[m]$ of length $s$. We denote the set of all such subsets by $\mathbf{K}_s$. 

\begin{lemma}\label{lem:fix2}
The fixed locus of the action of $\torus$ on $\widehat{Z}$ is a disjoint union $\bigsqcup_{\mathbf{j}}F_{\mathbf{j}}$ indexed by the set $\mathbf{K}_s$. Let $\iota_{F_{\mathbf{j}}}:F_{\mathbf{j}}\hookrightarrow \Gr(s,E^\vee)$ denote the corresponding component. Then $F_{\mathbf{j}}\cong B$ and $ S_s|_{F_{\mathbf{j}}}\cong \bigoplus_{k=1}^s L^\vee_{j_k}$. 
\end{lemma}

\begin{proof}
From the proof of Lemma \ref{lem:iso_basic_mutation}, we know 
\[
\widehat{Z}=\tot(Q_s^\vee \otimes \pi_s^*F^\vee \rightarrow \Gr(s, E^\vee)).
\]
The proof then follows similar to the proof of Lemma~\ref{lem:fix1}.

\end{proof}

\begin{corollary} \label{cor:fix2}
If $w_1,\dots, w_s$ are the equivariant Chern roots of $S_s^\vee$, then $\iota^*_{F_{\mathbf{j}}}(w_k)=-x_{j_k}$
\end{corollary}

Using the descriptions of the Lemmas~\ref{lem:fix1} and \ref{lem:fix2}, we have the following natural bijection of fixed loci under the Grassmannian duality
\[
	\mathbf{K}_r \Longleftrightarrow \mathbf{K}_s
\]
obtained by taking complements
\[
\mathbf{i}\longleftrightarrow \mathbf{i}^{\complement}.
\]

More preciously, if $F_{\mathbf{i}}\subset Z$ is a component of the fixed locus corresponding to the set $\mathbf{i}$, by abuse of notation, we denote by $\widehat{F}_{\mathbf{i}} \subseteq \widehat{Z}$ the component corresponding to the complement $\mathbf{i}^{\complement}$. Then we have a correspondence
\[
F_{\mathbf{i}} \longleftrightarrow \widehat{F}_{\mathbf{i}}.
\]

In order to prove Seiberg-like duality, we will show the correspondence of $I$-functions on each of the fixed point sets. To simplify notation, we will focus on the fixed locus $F_\mathbf{0}$ corresponding to the set $\mathbf{0}=\{1,\dots,r\}\subseteq [m]$. Then $\widehat{F}_{\mathbf{0}}$ corresponds with the set $\mathbf{0}^{\complement}=\{r+1,r+2,\dots r+s\}\subset [m]$, and we recall that $r+s=m$. 

The first step of the correspondence is to rewrite $i_{F_\mathbf{0}}^*(N_{\beta,\vec{d}})$ in a more useful form.
\begin{lemma} \label{lem:pullback1}
	For $1\leq j\leq r$, we set $a_j=d_j-\beta\cdot  x_j$, then we have
	\begin{align*}
		\iota_{F_\mathbf{0}}^*(N_{\beta,\vec{d}})=&\prod_{\substack{i,j=1\\ i\neq j}}^r 
		\frac{\prod_{h\leq a_j-a_i}(x_j-x_i+(\beta\cdot x_j-\beta\cdot x_i+h)z)}{\prod_{h\leq 0}(x_j-x_i+(\beta\cdot x_j-\beta\cdot  x_i+h)z)}\\
		\cdot & \prod_{j=1}^r  \frac{\prod_{k=1}^{n} \prod_{h=0}^{a_j-1} (z_k-x_j +(\beta\cdot z_k-\beta\cdot x_j - h) z)}{\prod_{i=1}^{m} \prod_{h=1}^{a_j} (x_j-x_i+(\beta\cdot x_j-\beta\cdot  x_i+h)z)}\\
		\cdot & \prod_{j=1}^r\prod_{i=1}^s\frac{\prod_{h\leq 0}(x_j-x_{r+i}+h z)}{\prod_{h\leq \beta\cdot x_j-\beta\cdot x_{r+i}}(x_j-x_{r+i}+h z)}\\
		\cdot & \prod_{j=1}^r\prod_{k=1}^{n} \frac{\prod_{h\leq 0}(z_k-x_j+h z)}{\prod_{h\leq -\beta\cdot x_j+\beta\cdot z_k}(z_k-x_j+h z)}
	\end{align*}
\end{lemma}

\begin{proof}
	By Corollary \ref{cor:fix1}, we have 
	\begin{align}
		\iota_{F_\mathbf{0}}^*(N_{\beta,\vec{d}})=&\prod_{\substack{i,j=1\\ i\neq j}}^r \frac{\prod_{h\leq d_j-d_i}(x_j-x_i+h z)}{\prod_{h\leq 0}(x_j-x_i+h z)} \cdot \prod_{j=1}^r \prod_{i=1}^m \frac{\prod_{h\leq 0} (x_j- x_i+hz)}{\prod_{h\leq d_j-\beta\cdot x_i} (x_j- x_i+hz)} \label{eq:Mfirstline}\\
		\cdot & \prod_{j=1}^r \prod_{k=1}^{n} \frac{\prod_{h\leq 0}(-x_j+z_k+hz)} {\prod_{h\leq -d_j+\beta\cdot z_k}(-x_j+z_k+hz)}. \label{eq:Msecondline}
	\end{align}
	Let $a_j=d_j-\beta\cdot x_j$, and notice that if $a_j< 0$, the second factor in the previous expression vanishes. So we replace $d_j$ with $a_j-\beta\cdot x_j$, and rewrite the previous expression. 
	
Consider the first factor of \eqref{eq:Mfirstline}. We can reindex the product with $h=\beta\cdot x_j - \beta\cdot x_i+h'$ to obtain
	\[
	\frac{\prod_{h\leq d_j-d_i}(x_j-x_i+h z)}{\prod_{h\leq 0}(x_j-x_i+h z)} =\frac{\prod_{h'\leq a_j-a_i}(x_j-x_i+(\beta\cdot x_j - \beta\cdot x_i +h') z)}{\prod_{h'\leq \beta\cdot x_i - \beta\cdot x_j}(x_j-x_i+(\beta\cdot x_j - \beta\cdot x_i + h') z)}.
	\]
	For $1\leq i \leq r$ with $i \neq j$, we do the same with the second factor to obtain 
	\begin{equation}\label{eq:MsecondFactor}
		\frac{\prod_{h\leq 0} (x_j- x_i+hz)}{\prod_{h\leq d_j-\beta\cdot x_i} (x_j- x_i+hz)}=
		\frac{\prod_{h'\leq \beta\cdot x_i - \beta\cdot x_j}(x_j-x_i+(\beta\cdot x_j - \beta\cdot x_i + h') z)}{\prod_{h'\leq a_j}(x_j-x_i+(\beta\cdot x_j - \beta\cdot x_i + h') z)}.    
	\end{equation}
	Combining the two, we obtain the contribution
	\[
	\frac{\prod_{h'\leq a_j-a_i}(x_j-x_i+(\beta\cdot x_j - \beta\cdot x_i +h') z)}{\prod_{h'\leq 0}(x_j-x_i+(\beta\cdot x_j - \beta\cdot x_i + h') z)}\cdot 
	\frac{1}{\prod_{h'=1}^{a_j}(x_j-x_i+(\beta\cdot x_j - \beta\cdot x_i + h') z)}.
	\]
	From our computation in \eqref{eq:MsecondFactor}, we see that when $i=j$, we obtain the contribution
	\[
	\frac{1}{\prod_{h'=1}^{a_j}(x_j-x_i+(\beta\cdot x_j - \beta\cdot x_i + h') z)}\Big{|}_{i=j}.
	\]
	
	Returning to the second factor of \eqref{eq:Mfirstline}, consider now the case when $i>r$. Then the second factor can be split into factors 
	\begin{align*}
		&\frac{\prod_{h\leq 0} (x_j- x_i+hz)} {\prod_{h\leq \beta\cdot x_j - \beta\cdot  x_i} (x_j- x_i+hz)} \cdot \frac{\prod_{h\leq \beta\cdot x_j - \beta\cdot x_i} (x_j- x_i+hz)}{\prod_{h\leq d_j-\beta\cdot x_i} (x_j- x_i+hz)} \\
		=& \frac{\prod_{h\leq 0} (x_j- x_i+hz)} {\prod_{h\leq \beta\cdot  x_j - \beta\cdot  x_i} (x_j- x_i+hz)}\cdot \frac{\prod_{h'\leq 0}(x_j-x_i+(\beta\cdot x_j - \beta\cdot x_i + h') z)}{\prod_{h'\leq a_j}(x_j-x_i+(\beta\cdot x_j - \beta\cdot x_i + h') z)}. 
	\end{align*}

    Recall $s=m-r$, so we replace $i$ with $r+k$ for $1\leq k\leq s$, and obtain the following contribution:
	\[
	\frac{\prod_{h\leq 0} (x_j- x_{r+k}+hz)}{\prod_{h\leq \beta\cdot x_j - \beta\cdot  x_{r+k}} (x_j- x_{r+k}+hz)}\cdot \frac{1}{\prod_{h'=1}^{a_j}(x_j-x_{r+k}+(\beta\cdot  x_j - \beta\cdot x_{r+k} + h') z)}
	\]
	
	Finally, the third factor in \eqref{eq:Msecondline}, we can rewrite as 
	\begin{align*}
		&\frac{\prod_{h\leq 0}(-x_j+z_k+hz)} {\prod_{h\leq -\beta\cdot x_j+\beta\cdot z_k}(-x_j+z_k+hz)}\cdot \frac{\prod_{h\leq -\beta\cdot x_j+\beta\cdot  z_k}(-x_j+z_k+hz)} {\prod_{h\leq -d_j+\beta\cdot z_k}(-x_j+z_k+hz)}\\
		=& \frac{\prod_{h\leq 0}(-x_j+z_k+hz)} {\prod_{h\leq -\beta\cdot x_j+\beta\cdot z_k}(-x_j+z_k+hz)}\cdot \frac{\prod_{h'\leq 0}(-x_j+z_k+(-\beta\cdot x_j+\beta\cdot z_k+h')z)} {\prod_{h'\leq -a_j}(-x_j+z_k+(-\beta\cdot x_j+\beta\cdot z_k +h') z)}\\
		=& \frac{\prod_{h\leq 0}(-x_j+z_k+hz)} {\prod_{h\leq -\beta\cdot x_j+\beta\cdot  z_k}(-x_j+z_k+hz)}\cdot \prod_{h'=0}^{a_j-1}(-x_j+z_k+(-\beta\cdot x_j+\beta\cdot z_k-h')z)
	\end{align*}
where in the second line, we used the substitution $h=-\beta\cdot x_j + \beta\cdot  z_k+h'$, 
and in the third line, since $a_j>0$ all factors appear in the numerator, so we substituted $h'$ for $-h'$. 
	
Putting it all together, we conclude. Notice we have combined all terms of the form $\prod_{h'=1}^{a_j}(x_j-x_i+(\beta\cdot x_j - \beta\cdot x_i + h) z)$ into one single product. 
\end{proof}

On the other side, we consider the corresponding fixed locus $ \widehat{F}_\mathbf{0} \subseteq \Gr(s,E^\vee)$ corresponding to the set $\mathbf{0}^\complement=\{r+1, \dots, m\}\subseteq [m]$. Then we can pull back $\widehat{N}_{\beta, \vec{d}}$ to obtain the following. 

\begin{lemma} \label{lem:pullback2}
For $1\leq j\leq s$, we set $a_j=d_j + \beta\cdot x_{r+j}$, then we have
\begin{align*}
\iota_{ \widehat{F}_\mathbf{0}}^*(\widehat{N}_{\beta,\vec{d}})=&\prod_{\substack{i,j=1\\ i\neq j}}^s 
\frac{\prod_{h\leq a_j-a_i}(x_{r+i}-x_{r+j}+ (\beta\cdot x_{r+i}-\beta\cdot  x_{r+j}+h)z)}{\prod_{h\leq 0}(x_{r+i}-x_{r+j}+ (\beta\cdot x_{r+i}-\beta\cdot  x_{r+j}+h)z)}\\
\cdot & \prod_{j=1}^s\frac{\prod_{k=1}^n\prod_{h=1}^{a_j}(-x_{r+j}+z_k+(\beta\cdot  z_k-\beta\cdot x_{r+j}+h) z)}{\prod_{i=1}^m \prod_{h=1}^{a_j}(-x_{r+j}+x_i+(\beta\cdot x_i-\beta\cdot x_{r+j}+h) z)}\\
\cdot & \prod_{i=1}^r \prod_{j=1}^s\frac{\prod_{h\leq 0}(x_i-x_{r+j}+h z)}{\prod_{h\leq \beta\cdot x_i-\beta\cdot x_{r+j}}(x_i-x_{r+j}+h z)}\\
\cdot & \prod_{i=1}^r\prod_{k=1}^n\frac{\prod_{h\leq 0}(z_k-x_i+h z)}{\prod_{h\leq -\beta\cdot x_i+\beta\cdot  z_k}(z_k-x_i+h z)}
\end{align*}
\end{lemma}

\begin{proof}
By Corollary \ref{cor:fix2}, we have
\begin{align}
\iota_{ \widehat{F}_\mathbf{0}}^*(\widehat{N}_{\beta,\vec{d}}) 
&= \prod_{\substack{i,j=1\\ i\neq j}}^s 
\frac{\prod_{h\leq d_j-d_i}(-x_{r+j}+x_{r+i}+hz)}{\prod_{h\leq 0}(-x_{r+j}+x_{r+i}+hz)} \cdot \prod_{j=1}^s \prod_{i=1}^m \frac{\prod_{h\leq 0} (-x_{r+j}+ x_i+hz)}{\prod_{h\leq d_j+\beta\cdot x_i} (-x_{r+j}+ x_i+hz)}\label{eq:Nfirstline}\\
& \cdot \prod_{j=1}^s\prod_{k=1}^{n}\frac{\prod_{h\leq d_j+\beta\cdot z_k}(-x_{r+j}+ z_k+hz)}{\prod_{h\leq 0}(-x_{r+j}+ z_k+hz)}\cdot \prod_{i=1}^m \prod_{k=1}^n \frac{\prod_{h\leq 0}(- x_i+ z_k+hz)}{\prod_{h\leq -\beta\cdot x_i+\beta\cdot z_k}(- x_i+ z_k+hz)}\label{eq:Nsecondline}
\end{align}
	
The second factor in the product vanishes unless $a_j=d_j+\beta\cdot x_{r+j}\geq 0$, for $1\leq j\leq s$. So we replace $d_j$ with $a_j-\beta\cdot x_{r+j}$ in order to write this in a more useful form.  
	
First, consider the first two factors of \eqref{eq:Nfirstline}. In the second factor, we first consider the case when $r+1\leq i\leq m$ and $i\neq r+j$. But we will reindex $i$ with $r+i$. We can reindex both products with $h=\beta\cdot x_{r+i} - \beta\cdot x_{r+j}+h'$ to obtain
\begin{align*}
&\frac{\prod_{h'\leq a_j-a_i}(x_{r+i}-x_{r+j}+(\beta\cdot x_{r+i} - \beta\cdot  x_{r+j} +h') z)}{\prod_{h'\leq \beta\cdot  x_{r+j} - \beta\cdot  x_{r+i}}(x_{r+i}-x_{r+j}+(\beta\cdot  x_{r+i} - \beta\cdot  x_{r+j} +h') z)}\\
&\quad \cdot
\frac{\prod_{h'\leq \beta\cdot  x_{r+j} - \beta\cdot  x_{r+i}}(x_{r+i}-x_{r+j}+(\beta\cdot  x_{r+i} - \beta\cdot  x_{r+j} +h') z)}{\prod_{h'\leq a_j}(x_{r+i}-x_{r+j}+(\beta\cdot  x_{r+i} - \beta\cdot  x_{r+j} +h') z)}    
\end{align*}
	
Combining the two, we obtain the contribution
\begin{align*}
& \frac{\prod_{h'\leq a_j-a_i}(x_{r+i}-x_{r+j}+(\beta\cdot  x_{r+i} - \beta\cdot  x_{r+j} + h') z)} {\prod_{h'\leq 0}(x_{r+i}-x_{r+j}+(\beta\cdot  x_{r+i} - \beta\cdot  x_{r+j} +h') z)} \\
&\cdot 
\frac{1}{\prod_{h'=1}^{a_j}(x_{r+i}-x_{r+j}+(\beta\cdot  x_{r+i} - \beta\cdot x_{r+j} +h') z)}
\end{align*}
	
Returning to the second factor of \eqref{eq:Nfirstline}, consider now the case when $1\leq i\leq r$. Then the second factor can be split into factors 
\begin{align*}
&\frac{\prod_{h\leq 0} (x_i- x_{r+j}+hz)} {\prod_{h\leq \beta\cdot  x_i - \beta\cdot  x_{r+j}} (x_i- x_{r+j}+hz)} \cdot \frac{\prod_{h\leq \beta\cdot  x_i - \beta\cdot  x_{r+j}} (x_i- x_{r+j}+hz)}{\prod_{h\leq d_j + \beta\cdot  x_i } (x_i- x_{r+j}+hz)} \\
=& \frac{\prod_{h\leq 0} (x_i- x_{r+j}+hz)} {\prod_{h\leq \beta\cdot  x_i - \beta\cdot  x_{r+j}} (x_i- x_{r+j}+hz)}\cdot \frac{1}{\prod_{h'=1}^{a_j}(x_{i}-x_{r+j}+(\beta\cdot  x_i - \beta\cdot  x_{r+j} +h') z)}.
\end{align*}
In the last computation, we made the substitution $h=\beta\cdot  x_i - \beta\cdot  x_{r+j}+h'$. 

A similar computation shows that when $i=r+j$, we obtain the following contribution from the second factor of \eqref{eq:Nfirstline}. 
\[
\frac{1}{\prod_{h'=1}^{a_j}(x_{i}-x_{r+j}+(\beta\cdot  x_i - \beta\cdot  x_{r+j} +h') z)}\Big{|}_{i = r+j}.
\]
	
Finally, we look at the remaining two factors in \eqref{eq:Nsecondline}. From the last factor, when $1\leq i\leq r$, we leave them as they are. For the rest, we can replace $i$ with $r+j$ for $1\leq j\leq s$, and write the last two factors as  
\[
\frac{\prod_{h\leq d_j+\beta\cdot z_k}(-x_{r+j}+z_k+hz)} {\prod_{h\leq -\beta\cdot x_{r+j}+\beta\cdot  z_k}(-x_{r+j}+z_k+hz)}.
\]
We make the substitution $h=h'-\beta\cdot x_{r+j}+\beta\cdot z_k$ to obtain
\[
\prod_{h=1}^{a_j} (-x_{r+j}+z_k+(-\beta\cdot  x_{r+j}+\beta\cdot z_k+h)z).
\]
	
Putting it all together, we arrive at the conclusion.
\end{proof}

Denote 
\begin{align*}
C_\beta :=  & \prod_{i=1}^r \prod_{k=1}^s\frac{\prod_{h\leq 0}(x_i-x_{r+k}+hz )}{\prod_{h\leq \beta\cdot x_i-\beta\cdot x_{r+k}}(x_i-x_{r+k}+hz )} \\  \cdot & \prod_{j=1}^r\prod_{k=1}^n\frac{\prod_{h\leq 0}(z_k-x_j+hz )}{\prod_{h\leq -\beta\cdot x_j+\beta\cdot  z_k}(z_k-x_j+hz )}.
\end{align*}

Following Lemmas \ref{lem:pullback1} and \ref{lem:pullback2} we can pull out several factors of $z$, and we rewrite 
\begin{align*}
\iota^*_{F_\mathbf{0}}(N_{\beta,\vec{d}})=&\prod_{\substack{i,j=1\\ i\neq j}}^r 
\frac{\prod_{h\leq a_j-a_i}z((\frac{x_j}{z}+\beta\cdot x_j)-(\frac{x_i}{z}+\beta\cdot  x_i)+h)}{\prod_{h\leq 0}z((\frac{x_j}{z}+\beta\cdot x_j)-(\frac{x_i}{z}+\beta\cdot x_i)+h)}\\
\cdot & \prod_{j=1}^r  \frac{\prod_{k=1}^n \prod_{h=0}^{a_j-1} z((\frac{z_k}{z}+\beta\cdot z_k)-(\frac{x_j}{z} +\beta\cdot x_j) - h)}{\prod_{i=1}^m \prod_{h=1}^{a_j} z((\frac{x_j}{z}+\beta\cdot x_j)-(\frac{x_i}{z}+\beta\cdot  x_i)+h)}
\cdot C_\beta
\end{align*}
under the substitution $a_j=d_j-\beta\cdot  x_j$ for $1\leq j\leq r$. Having done this, we define
\begin{align*}
N_{\beta,a} := z^{(n-m)a}\sum_{\sum a_i=a} &\left(\prod_{\substack{i,j=1\\ i\neq j}}^r 
\frac{\prod_{h\leq a_j-a_i}((\frac{x_j}{z}+\beta\cdot x_j)-(\frac{x_i}{z}+\beta\cdot x_i)+h)}{\prod_{h\leq 0}((\frac{x_j}{z}+\beta\cdot x_j)-(\frac{x_i}{z}+\beta\cdot x_i)+h)}\right.\\
&\cdot  \left.\prod_{j=1}^r  \frac{\prod_{k=1}^n\prod_{h=0}^{a_j-1} ((\frac{z_k}{z}+\beta\cdot z_k)-(\frac{x_j}{z} +\beta\cdot x_j) - h)}{\prod_{i=1}^m\prod_{h=1}^{a_j} ((\frac{x_j}{z}+\beta\cdot x_j)-(\frac{x_i}{z}+\beta\cdot x_i)+h)}\right)
\end{align*}

And similarly with 
\begin{align*}
\iota^*_{\widehat{F}_\mathbf{0}}(\widehat{N}_{\beta,\vec{d}})=&\prod_{\substack{i,j=1\\ i\neq j}}^s 
\frac{\prod_{h\leq a_j-a_i}z((\frac{x_{r+i}}{z} + \beta\cdot x_{r+i})-(\frac{x_{r+j}}{z}+\beta\cdot x_{r+j})+h)}{\prod_{h\leq 0}z(\frac{x_{r+i}}{z}+\beta\cdot x_{r+i})-(\frac{x_{r+j}}{z}+ \beta\cdot x_{r+j})+h)}\\
\cdot & \prod_{j=1}^s\frac{\prod_{k=1}^n \prod_{h=1}^{a_j} z( (\frac{z_k}{z}z_k+\beta\cdot z_k)-(\frac{x_{r+j}}{z}+\beta\cdot x_{r+j})+h)}{\prod_{i=1}^m \prod_{h=1}^{a_j}z((\frac{x_{i}}{z}+\beta\cdot x_i)-(\frac{x_{r+j}}{z}+\beta\cdot x_{r+j})+h)}\cdot C_\beta
\end{align*} 
under the substitution $a_j=d_j+\beta\cdot x_{r+j}$ for $1\leq j\leq s$. 
So we can write
\begin{align*}
\widehat{N}_{\beta,a} := z^{(n-m)a}\sum_{\sum a_i=a} &\left(\prod_{\substack{i,j=1\\ i\neq j}}^s 
\frac{\prod_{h\leq a_j-a_i}((\frac{x_{r+i}}{z} + \beta\cdot x_{r+i})-(\frac{x_{r+j}}{z}+\beta\cdot x_{r+j})+h)}{\prod_{h\leq 0}((\frac{x_{r+i}}{z}+\beta\cdot x_{r+i})-(\frac{x_{r+j}}{z}+ \beta\cdot x_{r+j})+h)}\right.\\
\cdot & \left.\prod_{j=1}^s\frac{\prod_{k=1}^n \prod_{h=1}^{a_j}( (\frac{z_k}{z}+\beta\cdot z_k)-(\frac{x_{r+j}}{z}+\beta\cdot x_{r+j})+h)}{\prod_{i=1}^m \prod_{h=1}^{a_j}((\frac{x_{i}}{z}+\beta\cdot x_i)-(\frac{x_{r+j}}{z}+\beta\cdot x_{r+j})+h)}
\right)
\end{align*}

\begin{proposition} \label{p.Gr_dual}
\begin{enumerate}
	\item \label{item:mn2} When $m \geq n+2$, then 
	\[
	N_{\beta,a}=\widehat{N}_{\beta,a}.
	\]
		
	\item When $m = n+1 $, then 
    \[
    N_{\beta,a}=\sum_{p=0}^a \frac{(-1)^{s(a-p)}}{(a-p)!} z^{p-a} \widehat{N}_{\beta,p}.
    \]
		
	\item When $m = n$, then
	\[
	N_{\beta, a}=\sum_{p=0}^a \Big(\prod_{h=0}^{p-1}  \frac{\sum_{k=1}^n (\frac{z_k}{z}+\beta\cdot z_k) - \sum_{k=1}^m (\frac{x_k}{z}+\beta\cdot x_k) + s -h}{p-h}\Big)(-1)^{sp} \widehat{N}_{\beta,a-p}
	\]
\end{enumerate}
\end{proposition}

\begin{proof}
It is enough to show when $F_{\bf{i}}= F_\mathbf{0}$ since, for other cases, we only need to shift the index. This proposition essentially appears in Hai Dong's dissertation \cite{Don}, which is not available online. However, part \ref{item:mn2} appears in \cite{BPZ}, and a full account together with complete proof the all three parts of this proposition without the factors of $z^{(n-m)a}$ appears in \cite{Zha} as Lemma~A.2, Corollary~A.3, and Lemma~A.4. 

More precisely, from Lemma~A.2, Corollary~A.3, and Lemma~A.4 in \cite{Zha}, after the changes of variables $\lambda_{f_i} \leftrightarrow \frac{x_{i}}{z}+\beta\cdot x_i$ and $\eta_k\leftrightarrow \frac{z_k}{z}+\beta\cdot z_k$, we obtain the following equations for $m\geq n+2$, for $m=n+1$, and for $m=n$, respectively:  
\begin{align*}
    &\frac{N_{\beta,a}}{z^{(n-m)a}}=\frac{\widehat{N}_{\beta,a}}{z^{(n-m)a}}\\
    &\frac{N_{\beta,a}}{z^{-a}}=\sum_{p=0}^a \frac{(-1)^{s(a-p)}}{(a-p)!} \frac{\widehat{N}_{\beta,p}}{z^{-p}}\\
    &N_{\beta, a}=\sum_{p=0}^a \Big(\prod_{h=0}^{p-1}  \frac{\sum_{k=1}^n (\frac{z_k}{z}+\beta\cdot z_k) - \sum_{k=1}^m (\frac{x_k}{z}+\beta\cdot x_k) + s -h}{p-h}\Big)(-1)^{sp} \widehat{N}_{\beta,a-p}
\end{align*}
In the last equation, notice that no extra powers of $z$ appear, since $z^{(n-m)a}=1$, when $m=n$. Parts (1) and (3) of the proposition follow immediately. For part (2), multiplying the middle equation through by $z^{-a}$, we obtain the result. 
\end{proof}

Proposition~\ref{p.Gr_dual} immediately leads to the following generalized mutation conjecture mentioned in the introduction.

\begin{theorem} \label{Them.building_block}
With the same notation as in Sections~\ref{ss:Ifunctions} and \ref{s.proof:Gr}, 
we have
\begin{enumerate}
	\item\label{item:mn+2} When $m \geq n+2$, then
	\begin{align*}
		I_{Z}(q, q_1) = I_{\widehat{Z}}(q', q'_1)
	\end{align*}
	under the change of variables 
	\begin{equation}\label{cov1}
	q'_1=q_1\quad\text{ and }\qquad q'=q\cdot q_1^{- \langle c_1(E), -\rangle}.
	\end{equation}
 
	\item\label{item:mn+1} When $m = n+1$, then
	\begin{align*}
		I_{Z}(q, q_1) = e^{(-1)^s q_1/z} I_{\widehat{Z}}(q', q_1')
	\end{align*}
	under the change of variables \eqref{cov1}.
		
	\item When $m = n$, then
	\begin{align*}
		I_{Z}(q, q_1) = (1+(-1)^{s} q_1)^{\Psi} I_{\widehat{Z}}(q', q_1')
	\end{align*}
	under the change of variables 
	\[
	q'_1=q_1\quad\text{ and }\qquad q'=\frac{q\cdot q_1^{- \langle c_1(E), -\rangle}}{(1+(-1)^{s} q_1)^{\langle c_1(F), - \rangle-\langle c_1(E), - \rangle}},
	\]
    where $\Psi=s+(c_1(E)-c_1(F))/z$ 
\end{enumerate}
\end{theorem}

\begin{proof}
To prove this theorem, we will localize at the fixed points of the $\torus$-action. We will show that the relations given above are satisfied when restricted to the fixed points. This fact together with the isomorphism of state spaces will then prove the result. 

First, we write 
	\[
	\iota^*_{F_\mathbf{0}}(I_{Z}(q,q_1)) = \sum_{\beta}\sum_{a\geq 0} N_{\beta,a}\cdot C_\beta \cdot J_\beta \cdot q^\beta q_1^{a+\sum_{i=1}^r \beta\cdot x_i}
	\]
	
	On the other hand, we have
	\[
	\iota^*_{\widehat{F}_\mathbf{0}}(I_{\widehat{Z}}(q',q'_1)) = \sum_{\beta}\sum_{a\geq 0} \widehat{N}_{\beta,a}\cdot C_\beta \cdot J_\beta \cdot (q')^\beta (q'_1)^{a-\sum_{j=1}^s \beta\cdot x_{r+j}}.
	\] 

And we notice that 
\begin{equation}\label{eq:q1}
		q_1^{a+\sum_{i=1}^r \beta\cdot x_i} =q_1^{a-\sum_{i=1}^s \beta\cdot x_{r+i}}\cdot q_1^{\sum_{i=1}^m \beta\cdot x_i}
	\end{equation}
Now we apply Proposition~\ref{p.Gr_dual}. When $m\geq n+2$, after applying Equation~\eqref{eq:q1}, the result is immediate. When $m = n+1$, we can write 
	\begin{align*}
		\iota^*_{F_\mathbf{0}}(I_{Z}(q,q_1)) &= \sum_{\beta}C_\beta \cdot J_\beta \cdot q^\beta q_1^{\sum_{i=1}^r \beta\cdot x_i}\sum_{a\geq 0} N_{\beta,a}\cdot q_1^{a}\\
		&=\sum_{\beta}C_\beta \cdot J_\beta \cdot q^\beta q_1^{\sum_{i=1}^r \beta\cdot x_i}\sum_{a\geq 0} \sum_{p=0}^a \frac{(-1)^{s(a-p)}}{(a-p)!} z^{p-a}\widehat{N}_{\beta,p} \cdot q_1^a\\
		&=\sum_{\beta}C_\beta \cdot J_\beta \cdot q^\beta q_1^{\sum_{i=1}^r \beta\cdot x_i}\sum_{p\geq 0}\sum_{a'\geq 0} \widehat{N}_{\beta,p} \cdot q_1^p \frac{(-1)^{s(a')}z^{-a'}}{(a')!} q_1^{a'}\\
        &=e^{(-1)^s q_1/z}\sum_{\beta}C_\beta \cdot J_\beta \cdot q^\beta\sum_{p\geq 0} \widehat{N}_{\beta,p} \cdot q_1^{p+\sum_{i=1}^r \beta\cdot x_i}.
	\end{align*}
After applying Equation~\eqref{eq:q1}, we obtain the result. 
	
When $m=n$, set $\Psi_\beta=\sum_{k=1}^n \beta\cdot z_k -\sum_{j=1}^m \beta\cdot x_j$, and recall $\Psi=\sum_{k=1}^n \tfrac{z_k}{z}-\sum_{j=1}^m \tfrac{x_j}{z}+s$. Then we have 
	\begin{align*}
		\iota^*_{F_\mathbf{0}}&(I_{Z}(q,q_1)) = \sum_{\beta}C_\beta \cdot J_\beta \cdot q^\beta q_1^{\sum_{i=1}^r \beta\cdot x_i}\sum_{a\geq 0} N_{\beta,a}\cdot q_1^{a}\\
		&=\sum_{\beta}C_\beta \cdot J_\beta \cdot q^\beta q_1^{\sum_{i=1}^r \beta\cdot x_i}\sum_{a\geq 0} \sum_{p=0}^a \Big(\prod_{h=0}^{p-1}  \frac{\Psi +\Psi_\beta -h}{p-h}\Big) \widehat{N}_{\beta,a-p}(-1)^{sp} q_1^a\\
		&=\sum_{\beta}C_\beta \cdot J_\beta \cdot q^\beta q_1^{\sum_{i=1}^r \beta\cdot x_i}\sum_{a\geq 0} \sum_{p'=0}^a \Big(\prod_{h=0}^{a-p'-1}  \frac{\Psi+\Psi_\beta -h}{(a-p')-h}\Big) \widehat{N}_{\beta,p'}(-1)^{s(a-p')}q_1^a\\
		&=\sum_{\beta}C_\beta \cdot J_\beta \cdot q^\beta q_1^{\sum_{i=1}^r \beta\cdot x_i}\sum_{p'\geq 0} \sum_{a\geq p'} \Big(\prod_{h=0}^{a-p'-1}  \frac{\Psi+\Psi_\beta -h}{(a-p')-h}\Big) \widehat{N}_{\beta,p'}(-1)^{s(a-p')}q_1^a\\
	    &=\sum_{\beta}C_\beta \cdot J_\beta \cdot q^\beta q_1^{\sum_{i=1}^r \beta\cdot x_i}\sum_{p'\geq 0} \widehat{N}_{\beta,p'}q_1^{p'} \sum_{a\geq p'} \Big(\prod_{h=0}^{a-p'-1}  \frac{\Psi +\Psi_\beta -h}{(a-p')-h}\Big) (-1)^{s(a-p')}q_1^{a-p'}\\
	    &=\sum_{\beta}C_\beta \cdot J_\beta \cdot q^\beta q_1^{\sum_{i=1}^r \beta\cdot x_i}\sum_{p'\geq 0} \widehat{N}_{\beta,p'}q_1^{p'} \sum_{a'\geq 0} \Big(\prod_{h=0}^{a'-1}  \frac{\Psi+\Psi_\beta -h}{a'-h}\Big) (-1)^{s a'} q_1^{a'}\\
	    &=\sum_{\beta}C_\beta \cdot J_\beta \cdot q^\beta q_1^{\sum_{i=1}^r \beta\cdot x_i}(1+(-1)^sq_1)^{\Psi+\Psi_\beta}\sum_{p'\geq 0} \widehat{N}_{\beta,p'}\cdot q_1^{p'}\\
         &=(1+(-1)^sq_1)^{\Psi}\sum_{\beta}C_\beta \cdot J_\beta \cdot q^\beta q_1^{\sum_{i=1}^m \beta\cdot x_i}(1+(-1)^sq_1)^{\Psi_\beta}\sum_{p'\geq 0} \widehat{N}_{\beta,p'}q_1^{p'-\sum_{j=1}^s \beta\cdot x_{r+j}}
\end{align*}
After changing the variables, we obtain the result. 
\end{proof}

\begin{remark}\label{Rem:new case}
    Note that the above theorem already implies new cases of the mutation conjecture in the absolute case. For example, our theorem implies the mutation conjecture for the following quiver, as we will explain:
    \begin{equation}\label{Qex}
	\begin{tikzpicture}
			\node[draw,
			circle,
			minimum size=0.6cm,
			] (gauge-1) at (0,0){1};
			
			\node [draw,
			circle,
			left=1.2cm of gauge-1
			]  (gauge-r) {$r$};

                \node [draw,
			minimum width=0.6cm,
			minimum height=0.6cm,
			left=1.2cm of gauge-r
			]  (frame-m) {$m$};

                \node [draw,
			circle,
			right=1.2cm of gauge-1
			]  (gauge-2) {$2$};
			
			\node [draw,
			minimum width=0.6cm,
			minimum height=0.6cm,
			]  (frame-4) at (0.9,1.3) {$4$};
			
                \draw[-stealth] (frame-m.east) -- (gauge-r.west)
			node[midway,above]{$1$};
   
			\draw[-stealth] (gauge-r.east) -- (gauge-1.west)
			node[midway,above]{$n$};
			
			\draw[-stealth] (gauge-1.north) -- (frame-4.south west)
			node[near end,left]{$2$};

                \draw[-stealth] (frame-4.south east) -- (gauge-2.north)
			node[midway,above]{$1$};

                \draw[-stealth] (gauge-2.west) -- (gauge-1.east)
			node[midway,below]{$4$};
	\end{tikzpicture}.
    \end{equation}
    Here a number over the arrow denotes the number of such arrows. 
    
    First consider the sub-quiver
     \begin{equation}\label{Qex2}
	\begin{tikzpicture}
			\node[draw,
			circle,
			minimum size=0.6cm,
			] (gauge-1) at (0,0){1};

                \node [draw,
			circle,
			right=1.2cm of gauge-1
			]  (gauge-2) {$2$};
			
			\node [draw,
			minimum width=0.6cm,
			minimum height=0.6cm,
			]  (frame-4) at (0.9,1.3) {$4$};

			\draw[-stealth] (gauge-1.north) -- (frame-4.south west)
			node[near end,left]{$2$};

                \draw[-stealth] (frame-4.south east) -- (gauge-2.north)
			node[midway,above]{$1$};

                \draw[-stealth] (gauge-2.west) -- (gauge-1.east)
			node[midway,below]{$4$};
	\end{tikzpicture}.
    \end{equation}
After choosing suitable stability conditions, the critical locus of the superpotential associated to \eqref{Qex2} yields the \emph{Gulliksen-Neg\r{a}rd 3-fold} (see e.g. \cite[Figure 5]{BPZ}), which we denote by $GN$.  This variety is described in more detail in Example~\ref{e:GN}.
    
If we then consider the  quiver \eqref{Qex}, the  critical locus of the superpotential associating to the whole quiver is equal to the relative quiver variety (with base  $B=GN$)  with quiver
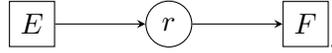
\begin{equation*}
	\begin{tikzpicture}
			\node[draw,
			circle,
			minimum size=0.6cm,
			] (gauge) at (0,0){$r$};

                \node[draw,
			minimum width=0.6cm,
			minimum height=0.6cm,
                left=1.2cm of gauge
			] (frame-1) {$E$};

                \node[draw,
			minimum width=0.6cm,
			minimum height=0.6cm,
                right=1.2cm of gauge
			] (frame-2) {$F$};

                \draw[-stealth] (frame-1.east) -- (gauge.west)
			node[midway,above]{};
   
			\draw[-stealth] (gauge.east) -- (frame-2.west)
			node[midway,above]{};
        \end{tikzpicture},
   \end{equation*}
   where $E=\op{tot}(\cO_{GN}^{\oplus m})$ and $F$ is the pullback of $\cO_{\PP^7}(-1)^{\oplus n}$ by the natural map $GN \to \PP^7$ determined by the quiver \eqref{Qex2}.  Theorem~\ref{Them.building_block} on relative mutations then gives the mutation conjecture for \eqref{Qex}.
  
\end{remark}

\section{PAX/PAXY correspondence} \label{s.general_paxpaxy}

We will apply the methods of the previous section to study $I$-functions of resolutions of determinantal varieties. As described in \cite{JKL} using the language of GLSMs, for a possibly singular determinantal variety $B(A,s)$, there are two different ways of obtaining a resolution of $B(A,s)$ as the critical locus of a certain function on a smooth non-compact variety. Although the two different resolutions are isomorphic via Grassmannian duality, the $I$-functions for the resolutions are different distinct, as the relative quasimap theories of the two models are significantly different (see Remark~\ref{r:redundant}). In this section, we show that the resulting $I$-functions for the resolutions are related in a similar fashion as the $I$-functions in the previous section. Indeed, we will show that the geometric relationship between the two LG models may also be described by mutation.

\subsection{Determinantal varieties and resolutions}

Let the setup be as before, i.e., $B$ is a smooth projective variety, and $ E=\bigoplus_{i=1}^m  L_i$ and $ F=\bigoplus_{i=1}^n  M_i$ are vector bundles on $B$ of rank $m$ and $n$, respectively with $m \geq n$. Assume that $ E^\vee \otimes  F$ is generated by global sections, and fix a generic section $A \in \mathcal{H}om( E,  F)$, i.e., at a generic point $b \in B$, the matrix $A_b:=A|_b$ is of full rank.

\begin{definition}
    Let $0 \leq s \leq n$, then the determinantal variety $B(A, s)$ is defined as
\begin{align*}
	B(A,s) := \left\{ b \in B \mid \mathrm{rank} A_b \leq s  \right\}.
\end{align*}
\end{definition}

It is well known that
\begin{enumerate}
	\item  Locally, the ideal $I(B(A, s))$ is generated by all the $(s + 1) \times (s + 1)$ (determinant) minors of $A$.
	\item If $B(A,s)$ is not empty, then the codimension of $B(A,s)$ in $B$ is
	\begin{align*}
		\mathrm{codim} B(A,s) = (n-s)(m-s).
	\end{align*}
	\item For $s \geq 1$, $B(A,s)$ is not a complete intersection.
	\item In general, determinantal varieties are singular. Under the assumptions above,
	\begin{align*}
		\mathrm{Sing}(B(A,s)) \subseteq B(A,s-1)  \subset B(A,s). 
	\end{align*}
\end{enumerate}
See \cite{Otta, JKL, Har} for more details.

\begin{example}\label{e:GN}
	Let $B=\mathbb{P}^N$
	and let $ E= \cO^{\oplus m}$ and $ F = \cO(1)^{\oplus n}$. Assume $m=n$. Under these conditions, the Calabi--Yau determinantal varieties 
	satisfy the following conditions:
	\begin{enumerate}
		\item Dimension: $\dim B(A,s) = N-(m-s)^2$.
		\item Calabi--Yau condition: $N+1-(m-s)m = 0$.
	\end{enumerate}
	In particular, in dimension 3, if we combine the above two conditions, we find that Calabi--Yau determinantal varieties of dimension 3 with $E$ and $F$ as above are classified by 
	\begin{align*}
		(s, m, N) = (4, 5, 4), \quad (2, 4, 7), \quad (1, 5, 19).
	\end{align*}
	The second case is called Gulliksen-Neg\r{a}rd 3-fold \cite{GN}, which is smooth (see also Remark~\ref{Rem:new case}).

\end{example}

There are two (related) desingularizations of $B(A,s)$, each defined in terms of the family of matrices $\{A_b \mid b \in B\}.$  The first, called the \emph{PAX} model, is based on the observation that $A_b$ is rank $s$ if and only if the kernel of $A_b$ is a subspace of $E_b$ of dimension $m-s$.  Informally, the PAX desingularization of $B(A, s)$ parametrizes all subspaces of $\ker A_b$ of dimension $r$ as $b$ varies over $B$.

The second desingularization of $B(A, s)$, called the \emph{PAXY} model, uses the fact that $A_b$ has rank at most $s$ if and only if, as a matrix, $A_b$ factors into a product of an $n\times s$ matrix $Y_b$ with an $s \times m$ matrix $X_b$.  Informally, the PAXY desingularization parametrizes all such factorizations of $A_b$ whose second factor $X_b$ is full rank. 
The formal definitions of these varieties below and the proof that they are in fact resolutions of $B(A,s)$ are included below.  This fact is well-known to experts, but we were unable to find a reference in the literature.

To construct the PAX resolution, let $r=m-s$, and consider the Grassmannian bundle $\Gr(r, E)$ of $r$-planes of fibers of $E$ (as in the previous sections, we are suppressing the subscript $B$ on this Grassmannian bundle, to ease notation) and let 
\[
\pi_r: \Gr(r, E) \to B
\]
denote the natural projection. Over $\Gr(r, E)$, we have the tautological sequence
\begin{align*}
	0 \longrightarrow S_r \longrightarrow \pi_r^*  E \longrightarrow Q \longrightarrow 0.
\end{align*}
Dualizing and tensoring with $\pi_r^*{ F}$, we obtain a map
\begin{equation}\label{taut}
    \pi_r^*  (E^\vee \otimes F) \to S_r^\vee \otimes \pi_r^* F.
\end{equation}
Denote by $\widetilde A$ the section of $S_r^\vee \otimes \pi_r^* F$ obtained by composing the above map with $\pi_r^*A \in \Gamma(\Gr(r, E), \pi_r^*(E^\vee \otimes F)).$ 
\begin{definition}
    Define the PAX resolution of $B(A, s)$ to be the zero locus
    \[
    Z_A = \{\widetilde A = 0\}.
    \]
\end{definition}

We will verify that $Z_A$ is a resolution of $B(A,s)$.  
We require a mild generalization of Bertini's theorem, which holds by the same argument.  For instance, the proof of \cite[Theorem~2.8]{Otta} extends verbatim to the following situation.
\begin{lemma}\label{p.Bertini}

    Let $V \rightarrow B$ be a vector bundle on a smooth  variety $B$.  Assume $W \subset H^0(B, F)$ is a subspace which generates $F$.  More precisely, we assume that there exists a set $s_1, \ldots, s_k$ of elements of  $W$ such that the map
    \[
    \oplus_{i=1}^k \mathcal O_B \to V
    \] 
    defined by $(s_1, \ldots, s_k)$ is surjective.
    Then for a generic section $s \in W$, the zero locus $\{s=0\} \subset B$ is smooth of codimension equal to $\op{rank}(V)$.  
\end{lemma}

\begin{proposition}\label{p.PAX_res}
    The PAX model $Z_A$ is a resolution of singularities of $B(A,s)$.
\end{proposition}

\begin{proof}
By the tautological sequence on $\Gr(r, E)$, the map \eqref{taut} is surjective.
We assume that $E^\vee \otimes F$ is generated by global sections, and therefore the image of the map 
\[
g: H^0( \Gr(r, E), \pi_r^*(E^\vee \otimes F)) \rightarrow H^0( \Gr(r, E), S_r^\vee \otimes \pi_r^*F)
\]
is a subspace of $H^0( \Gr(r, E), S_r^\vee \otimes \pi_r^*F)$ which generates $S_r^\vee \otimes \pi_r^*F$ in the sense of Lemma~\ref{p.Bertini}. The generic section $A \in H^0(B, E^\vee \otimes F) = H^0( \Gr(r, E), \pi_r^*(E^\vee \otimes F))$ induces a generic section of $\op{im}(g)$.  Lemma~\ref{p.Bertini} then implies that $Z_A$ is smooth.  It follows from the construction that the map 
\[
Z_A \to B(A,s)
\]
is an isomorphism  away from $B(A, s-1)$.  
\end{proof}

To construct the PAXY resolution, let $S_s$ denote the rank-$s$ tautological bundle on $\Gr(s, E^\vee)$ and consider the total space $\op{tot}(S_s \otimes \pi_s^*F)$, where
\[
\pi_s: \Gr(s, E^\vee) \to B.
\]

Let $p_s: \op{tot}(S_s \otimes \pi_s^*F) \to \Gr(s, E^\vee)$ denote the projection.  Pulling back the tautological sequence, we obtain a map 
\[
\sigma: p_s^*(S_s \otimes \pi_s^*F) \to p_s^*\pi_s^*(E^\vee \otimes F).
\]
Composing with the tautological section of
\[
p_s^*(S_s \otimes \pi_s^*F) \rightarrow \op{tot}(S_s \otimes \pi_s^*F),
\]
we obtain a section $\tilde{\sigma} \in \Gamma(\op{tot}(S_s \otimes \pi_s^*F), p_s^*\pi_s^*(E^\vee \otimes F)).$
\begin{definition}
    Define the PAXY resolution of $B(A, s)$ to be the zero locus
    $\widehat Z_A = \{p_s^*\pi_s^*A - \tilde{\sigma}=0\}.$
\end{definition}

\begin{proposition}\label{p.PAXY_res}
  The PAXY model $\widehat{Z}_A$ is a resolution of $B(A,s)$. In fact, under the isomorphism of dual Grassmannian bundles, we have $Z_A \cong \widehat{Z}_A $.
\end{proposition}
\begin{proof}
    The first statement follows from the second, or again from Lemma~\ref{p.Bertini}. To prove the second, notice that under Grassmannian duality the section $\widetilde{A}$ of $S_r^\vee \otimes \pi_r^*F \rightarrow \Gr(r, E)$ is identified with a section $\widetilde{A}'$ of $Q_s \otimes \pi_s^*F \rightarrow \Gr(s, E^\vee)$. The PAX model $Z_A$ may therefore be represented as the locus $\{\widetilde{A}'=0\}$ in $\Gr(s, E^\vee).$
    Consider the  tautological sequence
    \begin{equation*}
        \begin{tikzcd}
            0 \ar[r] & S_s\otimes\pi_s^*F \ar[r, "\iota"] & \pi_s^*(E^\vee \otimes F) \ar[d, "\pi_s"'] \ar[r, "q"] & Q_s\otimes \pi_s^*F \ar[r] & 0 \\
            & & \Gr(s, E^\vee) \ar[u, bend right, "\pi_s^*A"'] \ar[ru, "\widetilde{A}'"'] & &
        \end{tikzcd}
    \end{equation*}
    on $\Gr(s, E^\vee),$ and note that
     $\widetilde{A}'$ is the composition of $\pi_s^*A$ with $q$.  
     Therefore $Z_A$ is exactly the locus of points $x$ in $\Gr(s, E^\vee)$
     for $\pi_s^*A(x)$ lies in $\op{im}(\iota)$. 
    Let $t \in H^0(Z_A, S_s\otimes\pi_s^*F)$ denote the section induced by restricting $\pi_s^*A$ to $Z_A$.  This defines an embedding $Z_A \hookrightarrow \tot(S_s\otimes\pi_s^*F).$
     A local coordinate computation verifies that
     the image of this map is exactly $\widetilde{Z}_A.$

\end{proof}

\subsection{PAX/PAXY models as a relative quiver mutation} 
In this section we show that the PAX and PAXY models are related by mutation.

\begin{theorem}\label{thm.pax/paxy_mutation}
The PAX and PAXY models 
each arise as the critical locus of a superpotential defined on a quiver bundle over $B$.
Furthermore, the respective quiver bundles and their associated superpotentials are related by a relative quiver mutation.     
\end{theorem}

\begin{proof}
Consider the following quiver bundle---similar to the one defined in Section~\ref{s.twisted_Gr_bdl}, but now decorated with a dashed arrow representing the map $A$:

\begin{equation} \label{d:PAX}
\begin{tikzpicture}
\node[draw,
    circle,
    minimum size=0.6cm,
] (gauge) at (0,0){r};

\node [draw,
    minimum width=0.6cm,
    minimum height=0.6cm,
    left=1.2cm of gauge
]  (frame-E) {$E$};

\node [draw,
    minimum width=0.6cm,
    minimum height=0.6cm,
    above=1.2cm of gauge
]  (frame-F) {$F$};

\draw[-stealth] (gauge.west) -- (frame-E.east)
    node[midway,above]{$X$};

\draw[-stealth] (frame-F.south) -| (gauge.north)
    node[near end,right]{$P$};

\draw[-stealth,dashed] (frame-E.north) -- (frame-F.west)
    node[midway,left]{$A$};
\end{tikzpicture}.
\end{equation}
Here we use a dashed arrow to represent that the section $A \in \mathcal{H}om( E,  F)$ is \emph{fixed} and does not vary in moduli. This dashed arrow, therefore, does not affect the GIT quotient, but it does play a role in defining the superpotential on the quiver bundle (as it creates a cycle).

As in Section~\ref{s.twisted_Gr_bdl}, we consider the vector bundle $V = \mathcal{H}om(\cO_B \otimes \CC^r, E) \oplus \mathcal{H}om(F, \cO_B \otimes \CC^r)$, and define
\[
Z := V \sslash_{\theta_+} \GL_r,
\]
where $\theta_+$ is the determinant. In contrast to Section~\ref{s.twisted_Gr_bdl}, we  use the map $A$ to define a gauge invariant superpotential on $Z$:
\begin{align*}
	W_{PAX} = \tr (PAX).
\end{align*}

The critical locus of the superpotential function $\mathrm{Crit}(W_{PAX})$ is given by 
\[\left\{ AX = PA = \tr (P \partial_b A \cdot X ) = 0 \right\}.
\]
Over $AX=0$, $A$ and $\partial_b A \cdot X$ generate the normal bundle. Therefore, we have $P=0$. The critical locus becomes
\[
\left\{ AX = 0 \right\} \subset \Gr(r, E).
\]
Notice that $AX$ is the section of $S_r^\vee \otimes \pi_r^*F$ induced from the tautological quotient $\pi_r^*(E \otimes F) \rightarrow S_r^\vee \otimes \pi_r^* F $. It follows that $ \left\{ AX = 0 \right\}$ is exactly $Z_A$ from the last subsection, thus justifying the nomenclature ``PAX model.''

Let us now consider the quiver mutation of \eqref{d:PAX} at the gauge node. After relabelling ($P$ and $X$), we obtain the following quiver
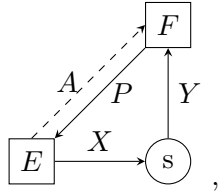
\begin{equation*}
\begin{tikzpicture}
\node[draw,
    circle,
    minimum size=0.6cm,
] (gauge) at (0,0){s};

\node [draw,
    minimum width=0.6cm,
    minimum height=0.6cm,
    left=1.2cm of gauge
]  (frame-E) {$E$};

\node [draw,
    minimum width=0.6cm,
    minimum height=0.6cm,
    above=1.2cm of gauge
]  (frame-F) {$F$};

\draw[-stealth] (frame-E.east) -- (gauge.west)
    node[midway,above]{$X$};

\draw[-stealth] (gauge.north) -| (frame-F.south)
    node[near end,right]{$Y$};

\draw[-stealth] (frame-F.south west) -- (frame-E.north east)
    node[midway,right]{$P$};

\draw[-stealth,dashed] (frame-E.north) -- (frame-F.west)
    node[midway,left]{$A$};
\end{tikzpicture},
\end{equation*}
with the character $\theta_-$, which is the inverse of the determinant.

There are now two loops, and we define the superpotential to be:
\begin{align*}
	W_{PAXY}= P(A-YX).
\end{align*}

The critical locus of the potential function $\mathrm{Crit}(W_{PAXY})$ is defined by
\[
\left\{ A- YX = XP = PY = \tr(P\partial_b A) = 0 \right\}.
\]
By similar considerations, we see that $\mathrm{Crit}(W_{PAXY})$ is exactly the locus $ \{ A- YX = 0 \} $. Notice that $YX$ is the section of $\pi_s^*(E \otimes F)$ induced from the tautological inclusion $S_s \otimes \pi_s^*F \hookrightarrow \pi_s^*(E^\vee \otimes F)$. We conclude that $ \{ A- YX = 0 \} $ is exactly the PAXY model $\widehat{Z}_A$ defined in the last subsection. 
\end{proof}

\subsection{A quantum Thom--Porteous formula}

In this section, we give $I$-functions for $Z_A$ and $\widehat{Z}_A$ in the general case that $B$ is a smooth projective variety. This may be viewed as a generalization of the quantum Lefschetz hyperplane theorem. 
\begin{remark}\label{r:redundant}
In light of Proposition~\ref{p.PAXY_res}, it might at first seem redundant to construct $I$-functions for both $Z_A$ and $\widehat{Z}_A$, however the different embeddings of $Z_A$ and $\widehat{Z}_A$ 
yield different $I$-functions. Heuristically, this is due the fact that the  relative quasimap invariants of \cite{Oh, SupTse} of the two GIT quotients have no obvious connection with one another.
\end{remark}

As before, let $J_B(q, u ,z)$ denote the $J$-function of $B$, and define $J_\beta = J_\beta(u, z)$ via 
\[
J_B(q, u, z) =\sum_\beta q^\beta J_\beta(u, z).
\]

Let $j_{PAX}: Z_A \rightarrow \Gr(r, E)$, and let $\pi_r: \Gr(r, E) \rightarrow B$. We follow the notations in Section \ref{s.generalized-Seiberg-duality} to denote the equivariant classes: let $x_i=c_1(L^\vee_i)$ and $\hat x_i:=\pi_r^*(x_i)$, let $z_j=c_1(M_j^\vee)$ and $\hat z_j=\pi_r^*(z_j)$, and let $\{ y_i \}_{i=1}^r$ be the Chern roots of $S^\vee_r$. Consider the following hypergeometric modification\footnote{In certain cases, such as when $B$ is Grassmannian, we recover the $I$-functions in \cite{HM18}.}
\begin{equation}\label{e:PAX}
	I_{Z_A}(q', q'_1, z)
 = j^*_{PAX} \left(\lim_{\lambda \mapsto 0} \sum_\beta (q')^\beta \sum_d (q'_1)^d\sum _{\substack{(d_1,\dots,d_r)\\ \sum d_i=d}} N^{PAX}_{\beta,\vec{d}} \cdot \pi_r^* J_\beta \right)
\end{equation}
with
\begin{align*}
	N^{PAX}_{\beta,\vec{d}}=&\prod_{\substack{i,j=1\\ i\neq j}}^r 
	\frac{\prod_{h \leq d_j-d_i}(y_j-y_i+hz)}{\prod_{h \leq 0}(y_j-y_i+hz)} \cdot \prod_{j=1}^r \prod_{i=1}^m \frac{\prod_{h \leq 0} (y_j-\hat x_i+hz)}{\prod_{h \leq d_j-\beta\cdot x_i} (y_j-\hat x_i+hz)}\\
	\cdot & \prod_{j=1}^r \prod_{k=1}^{n} \frac{\prod_{h \leq d_j-\beta\cdot z_k}(y_j-\hat z_k+hz)}{\prod_{h \leq 0}(y_j-\hat z_k+hz)},
\end{align*}
where $\lim_{\lambda \mapsto 0}$ denotes the nonequivariant limit of the $\torus$-equivariant cohomology.

Similar to the previous section, the map \begin{equation}\label{COV1}
    {q}^{\tilde \beta} \mapsto {q'}^{(\pi_r)_*(\tilde \beta)} {q_1'}^{\langle c_1(S_r^\vee), \tilde \beta \rangle},
\end{equation}
allows us to identify the Novikov ring  
$\QQ[\![\op{Eff}(\Gr(r, E))]\!]$ with a subring of 
$\QQ[\![\op{Eff}(B)]\!](\!(q_1')\!).$  Under this identification,
$I_{Z_A}(q', q_1', z)$ lies in the Givental space $H^*(Z_A)(\!(z^{-1})\!)$.

Then again by the work of \cite{CLS} (or by a generalization of \cite{Oh} to the twisted setting), one obtains:
\begin{proposition} \label{p.I_PAX}
The function
	$I_{Z_A}(q', q_1', z)$ lies in the Lagrangian cone of PAX model $Z_A$.
\end{proposition}
\begin{proof}
The proof follows the same argument as was used for Proposition \ref{p.I-function-Gr}. 
 
 The PAX model $Z_A$ is the zero locus of the regular section 
 \[
 AX \in \Gamma(\Gr(r, E), S_r^\vee\otimes \pi_r^* F).
 \]
Note that after a specialization of $t \mapsto 0$, the function $I_{PAX}(q, q_1, z)$ is precisely the nonequivariant $\lambda \mapsto 0$ limit (with respect to the $\torus$ action) of the Givental--Martin modification of the $(( S_r^\vee\otimes  F)_T, e_\torus(-))$-twisted $I$-function of $\Gr_T$. Here again $\Gr_T$ is the Abelian quotient $\mathcal{H}om(\CC^r, E)\sslash T$ and $( S_r^\vee\otimes  F)_T \to \Gr_T$ is the vector bundle corresponding to $ S_r^\vee\otimes F \to \Gr(r, E)$.  
 
By \cite[Theorem~1.11]{CLS} (after generalizing as in Proposition~\ref{p.I-function-Gr}),
the Givental--Martin modification of the $(( S_r^\vee\otimes  F)_T, e_\torus(-))$-twisted $I$-function of $\Gr_T$ lies on the  $(( S_r^\vee\otimes  F), e_\torus(-))$-twisted Lagrangian cone of $\Gr(r, E)$. We claim that the nonequivariant $\lambda \mapsto 0$ limit of this function is well defined. To see this we note that the second factor in $N^{PAX}_{\beta,\vec{d}}$ vanishes unless, for all $j \leq r$, there exists an $i \leq m$ such that
\[
d_j \geq \beta \cdot x_i,
\]
therefore the third sum in  \eqref{e:PAX} is in fact only over those $(d_1, \ldots, d_r)$ satisfying this property.  Furthermore, because $E^\vee \otimes F$ is ample, we also have that for all $i\leq m$, $k \leq n$, 
\[
\beta \cdot x_i \geq \beta \cdot z_k.
\]
Combining these inequalities, we conclude that for all $j \leq r$, $k \leq n$, 
\[
d_j \geq \beta \cdot z_k.
\]

Thus the factors 
\[
\prod_{j=1}^r \prod_{k=1}^{n} \frac{\prod_{h \leq d_j-\beta\cdot z_k}(y_j-\hat z_k+hz)}{\prod_{h \leq 0}(y_j-\hat z_k+hz)}
\] 
in the expression for $N^{PAX}_{\beta,\vec{d}}$ lie entirely in the numerator, and the nonequivariant limit of $N^{PAX}_{\beta,\vec{d}}$ is well-defined. The quantum Lefschetz theorem, as given in \cite{CoatesQLHT}, then implies the result. 
\end{proof}

For the PAXY model, let $j_{PAXY}: \widehat{Z}_A \rightarrow \op{tot}(S_s \otimes \pi_s^*F)$, and let $\pi_s : \Gr(s, E^\vee) \rightarrow B$. By abuse of notation, let $\hat x_i:=\pi_s^*(x_i)$, $\hat z_j=\pi_s^*(z_j)$, and let $\{ w_i \}_{i=1}^s$ be the Chern roots of $S^\vee_s$. We consider
\begin{align}\label{e:PAXY}
	I_{\widehat{Z}_A}(q, q_1, z)=j^*_{PAXY} \left(\lim_{\lambda \mapsto 0} \sum_\beta q^\beta \sum_d q_1^d\sum _{\substack{(d_1,\dots,d_s)\\ \sum d_i=d}} N^{PAXY}_{\beta,\vec{d}} \cdot p_s^*\pi_s^* J_\beta \right)
\end{align}
where
\begin{align*}
	N^{PAXY}_{\beta,\vec{d}}=&\prod_{\substack{i,j=1\\ i\neq j}}^s 
	\frac{\prod_{h \leq d_i-d_j}(w_i-w_j+hz)}{\prod_{h \leq 0}(w_i-w_j+hz)} \cdot \prod_{j=1}^s \prod_{i=1}^m \frac{\prod_{h \leq 0} (w_j+\hat x_i+hz)}{\prod_{h \leq d_j+\beta\cdot x_i} (w_j+\hat x_i+hz)}\\
	\cdot &\prod_{k=1}^{n}\prod_{j=1}^s\frac{\prod_{h \leq 0}(-w_j-\hat z_k+hz)}{\prod_{h \leq -d_j-\beta\cdot z_k}(-w_j-\hat z_k+hz)}\cdot \prod_{i=1}^m\prod_{k=1}^n \frac{\prod_{h \leq \beta\cdot x_i-\beta\cdot z_k}(\hat x_i-\hat z_k+hz)}{\prod_{h \leq 0}(\hat x_i-\hat z_k+hz)}.
\end{align*}
As before, the map \begin{equation}\label{COV2}
{q}^{\tilde \beta} \mapsto {q }^{\pi_*(\tilde \beta)} {q_1}^{\langle c_1(S_s^\vee), \tilde \beta \rangle},
\end{equation}
 allows us to identify the Novikov ring $\QQ[\![ \op{Eff}(\Gr(s, E^\vee))]\!]$ with a subring of $\QQ[\![\op{Eff}(B)]\!](\!(q_1)\!).$
Under this identification,
$I_{\widehat{Z}_A}(q, q_1, z)$ lies in the Givental space $H^*(\widehat{Z}_A)(\!(z^{-1})\!)$.

We have the following result for the PAXY $I$-function:
\begin{proposition} \label{p.I_PAXY}
The function
	$I_{\widehat{Z}_A}(q, q_1, z)$ lies in the Lagrangian cone of the PAXY model $\widehat{Z}_A$.
\end{proposition}

\begin{remark}
    Interestingly, Proposition~\ref{p.I_PAXY}  does not follow from the same argument as Proposition~\ref{p.I_PAX}, due to the fact that the vector bundle $S_s \otimes F$ is not concave. In particular, the nonequivariant limit with respect to the $\torus$-action is not \textit{a priori} well-defined. This was not an issue in Section~\ref{s.twisted_Gr_bdl} because we were working with the $\torus$-equivariant theory. 
    
    Nevertheless, the  GLSM description of the PAXY resolution suggests that this function should lie on the Lagrangian cone of $\widehat{Z}_A$. Remarkably, this fact will follow immediately from Proposition~\ref{p.I_PAX} together with our main comparison result below.
\end{remark}

The following theorem gives the PAX/PAXY correspondence for determinantal varieties.
\begin{theorem} \label{Thm.PAX/PAXY}
	Following the notations in \eqref{e:PAX} and \eqref{e:PAXY}, we have
	\begin{enumerate}
		\item When $m \geq n+2$, then
		\begin{align}\label{e1}
			I_{\widehat{Z}_A}(q, q_1) = I_{Z_A}(q', q_1').
		\end{align}
        under the change of variables 
    \begin{equation}\label{paxpaxycov}
            q'_1=q_1\quad\text{ and }\qquad q'=q\cdot q_1^{\langle c_1(E) , -\rangle}.
	 \end{equation}
		
		\item When $m = n+1$, then
		\begin{align}\label{e2}
			I_{\widehat{Z}_A}(q, q_1) = e^{(-1)^r q_1/z} I_{Z_A}(q', q_1').
		\end{align}
    under the change of variables \eqref{paxpaxycov}.
		
		\item When $m = n$, then
		\begin{align}\label{e3}
			I_{\widehat{Z}_A}(q, q_1) = (1+(-1)^{r} q_1)^{\Psi} I_{Z_A}(q', q_1').
		\end{align}
    under the change of variables 
	\[
	q'_1=q_1\quad\text{ and }\qquad q'=\frac{q\cdot q_1^{\langle c_1(E) , -\rangle}}{(1+(-1)^{r} q_1)^{\langle c_1(E), - \rangle-\langle c_1(F),-\rangle}},
	\]
    where $\Psi=(c_1(F)-c_1(E))/z+r$.

\end{enumerate}
\end{theorem}

\begin{remark}

Notice that the change of variables in this theorem is slightly different from that in Theorem~\ref{Them.building_block}.  This is because the $I$-functions for $Z_A$ and $\widehat{Z}_A$ are not the same as the $I$-functions appearing in Section~\ref{s.generalized-Seiberg-duality}.  The similarity between the change of variables in Theorem~\ref{Thm.PAX/PAXY} and Theorem~\ref{Them.building_block} is due to the fact that the equivariant lift of  $I_{Z_A}(q', q_1')$ (resp. $I_{\widehat{Z}_A}(q, q_1)$) closely resembles the $I$-function of $\widehat{Z}$ (resp. $Z$) after replacing $E$ with $E^\vee$ and $r$ with $s$.
\end{remark}

Proposition~\ref{p.I_PAXY} will follow from Theorem~\ref{Thm.PAX/PAXY}.

\begin{proof}[Proof of Proposition~\ref{p.I_PAXY}]
By Proposition~\ref{p.PAXY_res}, $Z_A \cong \widehat Z_A,$ so by Theorem~\ref{Thm.PAX/PAXY}, it suffices to show that the right hand side of equations~\eqref{e1},  \eqref{e2}, and \eqref{e3} lie in the Lagrangian cone of $Z_A$ in each case.

We first show that the change of variables \eqref{paxpaxycov}
is simply a result of how we chose coordinates for the respective Novikov rings.  Namely, the canonical isomorphism
$\Gr(r, E) = \Gr(s, E^\vee),$
induces an identification 
\[
\QQ[\![ \op{Eff}(\Gr(r, E))]\!] = \QQ[\![ \op{Eff}(\Gr(s, E^\vee))]\!].
\]

By \eqref{COV1} and \eqref{COV2},  for $\tilde \beta \in H_2(\Gr(r, E)) = H_2(\Gr(s, E^\vee))$, we can write $q^{\tilde \beta}$ as either 
$ {q '}^{\pi_*(\tilde \beta)} {q_1'}^{\langle c_1(S_r^\vee), \tilde \beta \rangle}$
or 
${q}^{\pi_*(\tilde \beta)} {q_1}^{\langle c_1(S_s^\vee), \tilde \beta \rangle}.$
We will show that the composition 
\[
{q}^{\pi_*(\tilde \beta)} {q_1}^{\langle c_1(S_s^\vee), \tilde \beta \rangle} \mapsto  q^{\tilde \beta} \mapsto {q'}^{\pi_*(\tilde \beta)} {q_1 '}^{\langle c_1(S_r^\vee), \tilde \beta \rangle}
\] 
is exactly the change of variables \eqref{paxpaxycov}.
The tautological sequence on $\Gr(r, E)$ is 
\[
0 \to S_r \to E \to S_s^\vee \to 0,
\]
i.e., the tautological bundle on $\Gr(s, E^\vee)$ is canonically identified with the quotient bundle on $\Gr(r, E)$.  Consequently, 
\[
c_1(S_s^\vee) - c_1(S_r^\vee) = c_1(E).
\]
Therefore
\begin{align*}
{q}^{\pi_*(\tilde \beta)} {q_1}^{\langle c_1(S_s^\vee), \tilde \beta \rangle}
 = & 
{q}^{\pi_*(\tilde \beta)} {q_1}^{\langle c_1(E), \tilde \beta \rangle + \langle c_1(S_r^\vee), \tilde \beta \rangle}
\\ = & 
{q }^{\pi_*(\tilde \beta)} {q_1}^{-\sum_{i=1}^m\langle x_i, \tilde \beta \rangle}{q_1}^{ \langle c_1(S_r^\vee), \tilde \beta \rangle} \\ = & 
{q'}^{\pi_*(\tilde \beta)} {q_1 '}^{\langle c_1(S_r^\vee), \tilde \beta \rangle}.
\end{align*}
under the change of variables \eqref{paxpaxycov}.  From this we see that the change of variables \eqref{paxpaxycov} is just an artifact of whether we are viewing $Z_A \cong \widehat Z_A$ as a subvariety of  $\Gr(r, E)$ or $\Gr(s, E^\vee)$.
Equation~\eqref{e1} then implies the result when $m\geq n+2$.

The remaining cases will follow from the string and divisor equations.  Recall that the (big) $J$-function of $B$, $
J_B(q, u, z)$, depends on choice of class $u \in H^*(B)$ and so, therefore, do  $I_{Z_A}(q', q_1')$ and $I_{\widehat{Z}_A}(q, q_1)$.  Let us write $I_{Z_A}(q', q_1', u)$ and $I_{\widehat{Z}_A}(q, q_1, u)$ to make this dependence explicit.

The string equation implies that, if $u = u_0 \mathbf{1}+ \tilde u$, where $\mathbf{1}$ is the fundamental class in $H^0(B)$, then 
\[
J_B(q, u, z) =e^{u_0/z}\sum_\beta q^\beta J_\beta(\tilde u, z).
\]
Equation \eqref{e2} may therefore be written as 
\[
I_{\widehat{Z}_A}(q, q_1, u) = I_{Z_A}(q', q_1', (-1)^rq_1 \mathbf{1} + u).
\]
The right hand side lies in the Lagrangian cone of $Z_A$ by Proposition~\ref{p.I_PAX}, which implies the result for $m=n+1$.

The divisor equation implies that, if $u = u_2 + \tilde u$ with $u_2 \in H^2(B)$, then 
\[
J_B(q, u, z) =e^{u_2/z}\sum_\beta q^\beta e^{\langle u_2, \beta\rangle}  J_\beta(\tilde u, z).
\]
Equation~\eqref{e3} may then be written as
\[
I_{\widehat{Z}_A}(q, q_1, u) =  I_{Z_A}\left(q', q_1', f(q_1)\left(\textstyle{r\cdot \mathbf{1} +\sum_{j=1}^m x_j - \sum_{k=1}^n z_k}\right)+ u\right),
\]
with $f(q_1) = \ln(1+(-1)^rq_1).$ Proposition~\ref{p.I_PAX} now implies the result when $m=n$
\end{proof}

\subsection{The proof of Theorem \ref{Thm.PAX/PAXY}}

Let 
\[
\widetilde I_{Z_A} = 
\sum_\beta q^\beta \sum_d q_1^d\sum _{\substack{(d_1,\dots,d_r)\\ \sum d_i=d}} N^{PAX}_{\beta,\vec{d}} \cdot \pi_r^* J_\beta
\]
and 
\[
\widetilde I_{\widehat{Z}_A} = 
\sum_\beta q^\beta \sum_d q_1'^d\sum _{\substack{(d_1,\dots,d_s)\\ \sum d_i=d}} N^{PAXY}_{\beta,\vec{d}} \cdot p_s^*\pi_s^* J_\beta ,
\]
so that $I_{Z_A} = j_{PAX}^* \lim_{\lambda \mapsto 0} \widetilde I_{Z_A}$ and similarly for $I_{\widehat{Z}_A}$.  

Rather than comparing $I_{Z_A}$ and $I_{\widehat{Z}_A}$ directly, we will instead compare $\pi^* \widetilde I_{Z_A} $ and $ \widetilde I_{\widehat{Z}_A},$
where $\pi$ is the composition 
\[
\pi: \op{tot}(S_s \otimes \pi_s^* F) \to \Gr(s, E^\vee) \xrightarrow{\cong} \Gr(r, E).
\]
We will show that  $\pi^* \widetilde I_{Z_A} $ coincides with $\widetilde I_{\widehat{Z}_A}$ under the specified change of variables and prefactor.

Then, by the commutative diagram:
\[
\begin{tikzcd}
	\widehat Z_A \ar[r, "j_{PAXY}"] \ar[d, "\cong"] & \op{tot}(S_s \otimes  F) \ar[r] \ar[dr, "\pi"] & \Gr(s, E^\vee) \ar[d, "\cong"] \\
	Z_A  \ar[rr, "j_{PAX}"] & & \Gr(r, E),
\end{tikzcd}
\]
it follows that  \[I_{Z_A} = j_{PAX}^* \lim_{\lambda \mapsto 0}\widetilde I_{Z_A} = j_{PAXY}^* \lim_{\lambda \mapsto 0}\widetilde I_{\widehat{Z}_A} = I_{\widehat{Z}_A}.\]

We parallel the argument of Section \ref{s.proof:Gr} by restricting to the fixed points of the $\torus$-action. We first write each of the modification factors in a more useful form.

\begin{lemma} \label{l.PAX_pullback}
	For $1 \leq j \leq r$, let $a_j = d_j - \beta \cdot x_j$, then we have
	\begin{align*}
		\iota_{F_\mathbf{0}}^*(N^{PAX}_{\beta,\vec{d}})=&\prod_{\substack{i,j=1\\ i\neq j}}^r \frac{\prod_{h \leq a_j - a_i}(x_j-x_i+(\beta\cdot x_j-\beta\cdot x_i+h)z)}{\prod_{h \leq 0}(x_j-x_i+(\beta\cdot x_j-\beta\cdot x_i+h)z)}\\
		\cdot & \prod_{j=1}^r  \frac{\prod_{k=1}^n\prod_{h=1}^{a_j} (x_j- z_k+(\beta\cdot x_j-\beta\cdot z_k +h) z)}{\prod_{i=1}^m\prod_{h=1}^{a_j} (x_j-x_i+(\beta\cdot x_j-\beta\cdot x_i+h)z)}\\
		\cdot & \prod_{j=1}^r\prod_{i=1}^s\frac{\prod_{h\leq 0}(x_j-x_{r+i}+h z)}{\prod_{h\leq \beta\cdot x_j-\beta\cdot x_{r+i}}(x_j-x_{r+i}+h z)}\\
		\cdot & \prod_{j=1}^r\prod_{k=1}^n\frac{\prod_{h\leq \beta\cdot x_j-\beta\cdot z_k}(x_j-z_k+h z)}{\prod_{h\leq 0}(x_j-z_k+h z)}.
	\end{align*}
\end{lemma}
\begin{proof}
	By Lemma \ref{lem:fix1}, we have   
	\begin{align*}
		\iota_{F_\mathbf{0}}^*(N^{PAX}_{\beta, \vec{d}})=&\prod_{\substack{i,j=1\\ i\neq j}}^r \frac{\prod_{h\leq d_j-d_i}(x_j-x_i+h z)}{\prod_{h\leq 0}(x_j-x_i+h z)} \cdot \prod_{j=1}^r \prod_{i=1}^m \frac{\prod_{h\leq 0} (x_j- x_i+hz)}{\prod_{h\leq d_j-\beta\cdot x_i} (x_j- x_i+hz)} \\
		\cdot & \prod_{j=1}^r \prod_{k=1}^{n} \frac{\prod_{h\leq d_j-\beta\cdot z_k}(x_j-z_k+hz)} {\prod_{h\leq 0}(x_j-z_k+hz)}.
	\end{align*}
	Exactly the same computations as the proof of Corollary \ref{lem:pullback1} can be used on the first two factors. Similar computation can be used on the last factor by writing 
	\begin{align*}
		&\frac{\prod_{h\leq \beta\cdot x_j-\beta\cdot z_k}(x_j- z_k+hz)}{\prod_{h\leq 0}(x_j- z_k+hz)}\cdot \frac{\prod_{h\leq d_j-\beta\cdot z_k}(x_j- z_k+hz)}{\prod_{h\leq \beta\cdot x_j-\beta\cdot z_k}(x_j- z_k+hz)} \\ 
		=&\frac{\prod_{h\leq \beta\cdot x_j-\beta\cdot z_k}(x_j- z_k+hz)}{\prod_{h\leq 0}(x_j- z_k+hz)}\cdot \prod_{h=1}^{a_j}(x_j- z_k+(\beta\cdot x_j-\beta\cdot z_k+h)z).
	\end{align*}
	Putting it all together, we obtain the conclusion. 
\end{proof}

Now we turn our attention to the PAXY model and rewrite the modification factor in a similar manner. 

\begin{lemma}\label{l.PAXY_pullback}
	For $1\leq j\leq s$, we set $a_j=d_j + \beta\cdot x_{r+j}$, then we have
	\begin{align*}
		\iota_{\widehat{F}_\mathbf{0}}^*(N^{PAXY}_{\beta, \vec{d}})=&\prod_{\substack{i,j=1\\ i\neq j}}^s \frac{\prod_{h \leq a_j-a_i}(x_{r+i}-x_{r+j}+(\beta\cdot x_{r+i}-\beta\cdot x_{r+j}+h)z)}{\prod_{h \leq 0}(x_{r+i}-x_{r+j}+(\beta\cdot x_{r+i}-\beta\cdot x_{r+j}+h)z)}\\
		\cdot & \prod_{j=1}^s  \frac{\prod_{k=1}^n\prod_{h=0}^{a_j-1} (x_{r+j}- z_k+(\beta\cdot x_{r+j}-\beta\cdot z_k - h) z)}{\prod_{i=1}^m\prod_{h=1}^{a_j} (x_i-x_{r+j}+(\beta\cdot x_i-\beta\cdot x_{r+j}+h)z)}\\
		\cdot & \prod_{j=1}^s\prod_{i=1}^r\frac{\prod_{h\leq 0}(x_i-x_{r+j}+h z)}{\prod_{h\leq \beta\cdot x_i-\beta\cdot x_{r+j}}(x_i-x_{r+j}+h z)}\\
		\cdot & \prod_{i=1}^r\prod_{k=1}^n\frac{\prod_{h\leq \beta\cdot x_i-\beta\cdot z_{k}}(x_i-z_k+h z)}{\prod_{h\leq 0}(x_i-z_k+h z)}.
	\end{align*}
\end{lemma}
\begin{proof}
	By Lemma \ref{lem:fix2}, we have
	\begin{align}
		\iota^*_{\widehat{F}_\mathbf{0}} (N^{PAXY}_{\beta, \vec{d}}) 
            =& 
        \prod_{\substack{i,j=1\\ i\neq j}}^s 
		\frac{\prod_{h\leq d_j-d_i}(-x_{r+j}+x_{r+i}+hz)}{\prod_{h\leq 0}(-x_{r+j}+x_{r+i}+hz)} 
            \cdot 
        \prod_{j=1}^s \prod_{i=1}^m \frac{\prod_{h\leq 0} (-x_{r+j}+ x_i+hz)}{\prod_{h\leq d_j+\beta\cdot x_i} (-x_{r+j}+ x_i+hz)} \nonumber \\
		\cdot &
        \prod_{k=1}^{n}\prod_{j=1}^s\frac{\prod_{h\leq 0}(x_{r+j}- z_k+hz)}{\prod_{h\leq -d_j-\beta\cdot z_k}(x_{r+j}- z_k+hz)}
            \cdot 
        \prod_{i=1}^m\prod_{k=1}^n \frac{\prod_{h\leq \beta\cdot x_i-\beta\cdot z_k}( x_i- z_k+hz)}{\prod_{h\leq 0}( x_i- z_k+hz)}. \nonumber 
	\end{align}
	The first two factors can be dealt with as the proof of Lemma~\ref{lem:pullback2}. Finally, we look at the last two factors in the second line of the above formula. From the last factor, when $1\leq i\leq r$, we leave them as they are. For the rest, we can replace $i$ with $r+j$ for $1\leq j\leq s$, and write the last two factors as  
	\[
	\frac{\prod_{h\leq \beta\cdot x_{r+j} - \beta\cdot z_k}(x_{r+j}-z_k+hz)}{\prod_{h\leq -d_j-\beta\cdot z_k}(x_{r+j} - z_k+hz)}.
	\]
	We make the substitution $h=h'-\beta\cdot x_{r+j}+\beta\cdot z_{k}$ to obtain
	\[
	\prod_{h=0}^{a_j-1} (x_{r+j} - z_k+(\beta\cdot x_{r+j} - \beta\cdot z_k - h)z).
	\]
	Putting it all together, we arrive at the conclusion.
\end{proof}

Again, similar to what was done in the previous section, we denote 
\begin{align*}
\widetilde C_{\beta} :=  & \prod_{j=1}^s\prod_{i=1}^r\frac{\prod_{h\leq 0}(x_i-x_{r+j}+hz)}{\prod_{h\leq \beta\cdot x_i-\beta\cdot x_{r+j}}(x_i-x_{r+j}+hz)}\\
		\cdot & \prod_{i=1}^r\prod_{k=1}^n\frac{\prod_{h\leq \beta\cdot x_i-\beta\cdot z_{k}}(x_i-z_k+hz)}{\prod_{h\leq 0}(x_i-z_k+hz)}.
\end{align*}

Following Lemmas \ref{l.PAX_pullback} and \ref{l.PAXY_pullback} we can write 
\begin{align*}
\iota_{\widehat{F}_\mathbf{0}}^*(N^{PAXY}_{\beta, \vec{d}})=&\prod_{\substack{i,j=1\\ i\neq j}}^s \frac{\prod_{h \leq a_j-a_i} z(-(\tfrac{x_{r+j}}{z}+\beta\cdot x_{r+j})+(\tfrac{x_{r+i}}{z}+\beta\cdot x_{r+i})+h)}{\prod_{h \leq 0} z(-(\tfrac{x_{r+j}}{z}+\beta\cdot x_{r+j})+(\tfrac{x_{r+i}}{z}+\beta\cdot x_{r+i})+h)}\\
		\cdot & \prod_{j=1}^s  \frac{\prod_{k=1}^n\prod_{h=0}^{a_j-1} z(-( \tfrac{z_k}{z}+\beta\cdot z_k)+(\tfrac{x_{r+j}}{z}+\beta\cdot x_{r+j}) - h)}{\prod_{i=1}^m\prod_{h=1}^{a_j} z(-(\tfrac{x_{r+j}}{z}+\beta\cdot x_{r+j}) + (\tfrac{x_i}{z}+\beta\cdot x_i)+h)}\cdot \widetilde C_{ \beta}
\end{align*}
under the substitution $a_j=d_j+\beta\cdot  x_{r+j}$ for $1\leq j\leq s$. Having done this, we define
\begin{align*}
N^{PAXY}_{ \beta,a} := z^{(n-m)a}\sum_{\sum a_i=a} &\left( \prod_{\substack{i,j=1\\ i\neq j}}^s \frac{\prod_{h \leq a_j-a_i} (-(\tfrac{x_{r+j}}{z}+\beta\cdot x_{r+j})+(\tfrac{x_{r+i}}{z}+\beta\cdot x_{r+i})+h)}{\prod_{h \leq 0} (-(\tfrac{x_{r+j}}{z}+\beta\cdot x_{r+j})+(\tfrac{x_{r+i}}{z}+\beta\cdot x_{r+i})+h)}\right.\\
	\cdot & \left.\prod_{j=1}^s  \frac{\prod_{k=1}^n\prod_{h=0}^{a_j-1} (-( \tfrac{z_k}{z}+\beta\cdot z_k)+(\tfrac{x_{r+j}}{z}+\beta\cdot x_{r+j}) - h)}{\prod_{i=1}^m\prod_{h=1}^{a_j} (-(\tfrac{x_{r+j}}{z}+\beta\cdot x_{r+j}) + (\tfrac{x_i}{z}+\beta\cdot x_i)+h)} 
 \right).
\end{align*}
And similarly with 
\begin{align*}
\iota_{F_\mathbf{0}}^*(N^{PAX}_{\beta,\vec{d}})=&\prod_{\substack{i,j=1\\ i\neq j}}^r \frac{\prod_{h \leq a_j - a_i} (-(\tfrac{x_i}{z}+\beta\cdot x_i)+(\tfrac{x_j}{z}+\beta\cdot x_j)+h)}{\prod_{h \leq 0} (-(\tfrac{x_i}{z}+\beta\cdot x_i)+(\tfrac{x_j}{z}+\beta\cdot x_j)+h)}\\
		\cdot & \prod_{j=1}^r  \frac{\prod_{k=1}^n\prod_{h=1}^{a_j} (-( \tfrac{z_k}{z}+\beta\cdot z_k)+(\tfrac{x_j}{z}+\beta\cdot x_j) +h)}{\prod_{i=1}^m\prod_{h=1}^{a_j} (-(\tfrac{x_i}{z}+\beta\cdot x_i)+(\tfrac{x_j}{z}+\beta\cdot x_j)+h)}
\end{align*} 
under the substitution $a_j=d_j-\beta\cdot L_{j}^\vee$ for $1\leq j\leq s$. 
So we can write
\begin{align*}
N^{PAX}_{\beta,a} :=  z^{(n-m)a}\sum_{\sum a_i=a} & \left(\prod_{\substack{i,j=1\\ i\neq j}}^r \frac{\prod_{h \leq a_j - a_i} (-(\tfrac{x_i}{z}+\beta\cdot x_i)+(\tfrac{x_j}{z}+\beta\cdot x_j)+h)}{\prod_{h \leq 0} (-(\tfrac{x_i}{z}+\beta\cdot x_i)+(\tfrac{x_j}{z}+\beta\cdot x_j)+h)}\right.\\
		\cdot & \left. \prod_{j=1}^r  \frac{\prod_{k=1}^n\prod_{h=1}^{a_j} (-(\tfrac{z_k}{z} +\beta\cdot z_k)+(\tfrac{x_j}{z}+\beta\cdot x_j) +h)}{\prod_{i=1}^m\prod_{h=1}^{a_j} (-(\tfrac{x_i}{z}+\beta\cdot x_i)+(\tfrac{x_j}{z}+\beta\cdot x_j)+h)}
\right).
\end{align*}

\begin{proposition} \label{p.pax/paxy}
\begin{enumerate}
	\item \label{item:mn22} When $m \geq n+2$, then 
	\[
	N^{PAXY}_{\beta,a}=N^{PAX}_{\beta,a}.
	\]
		
	\item When $m = n+1 $, then 
	\[
	N^{PAXY}_{\beta,a}=\sum_{p=0}^a \frac{(-1)^{r(a-p)}}{(a-p)!}z^{p-a} N^{PAX}_{\beta,p}.
	\]
		
	\item When $m = n$, then
	\[
	N_{\beta, a}^{PAXY}=\sum_{p=0}^a \Big(\prod_{h=0}^{p-1}  \frac{\sum_{k=1}^n -(z_k+\beta\cdot z_k) + \sum_{k=1}^m (x_k+\beta\cdot x_k) + r -h}{p-h}\Big)(-1)^{rp} N^{PAX}_{\beta,a-p}.
	\]
\end{enumerate}
\end{proposition}

\begin{proof}
This proof follows the proof of Proposition~\ref{p.Gr_dual} almost verbatim with the following changes: 
\begin{itemize}
    \item replace $N_{\beta,a}$ with $N_{\beta, a}^{PAXY}$; 
    \item replace $\widehat N_{\beta, a}$ with $N_{\beta,a}^{PAX}$; 
    \item replace $s$ with $r$; 
    \item from Lemma~A.2, Corollary~A.3, and Lemma~A.4 in \cite{Zha}, make the change of variables $\lambda_{f_i} \leftrightarrow -(\frac{x_{i}}{z}+\beta\cdot x_i)$ and $\eta_k\leftrightarrow -(\frac{z_k}{z}+\beta\cdot z_k)$.
\end{itemize} 

\end{proof}

The proof of Theorem~\ref{Thm.PAX/PAXY} follows almost immediately from Proposition~\ref{p.pax/paxy}. We will give just a few more details where there might be some chance for confusion. 

\begin{proof}[proof of Theorem~\ref{Thm.PAX/PAXY}]
As before, we  localize at the fixed points of the $\torus$-action. We will show that the $I$ functions restricted to the fixed points are related as described in the Theorem.

As in Lemma~\ref{l.PAXY_pullback}, we set $d_j=a_j-\beta\cdot x_{r+j}$ to obtain
\[
\iota^*_{\widehat{F}_\mathbf{0}}(\widetilde I_{\widehat{Z}_A}) = 
\sum_\beta \widetilde C_\beta \cdot J_\beta\cdot q^\beta \sum _{a\geq 0} N^{PAXY}_{\beta,a}  q_1^{a-\sum_{j=1}^s \beta\cdot x_{r+j}}
\]
and as in Lemma~\ref{l.PAX_pullback}, we set $d_j=a_j+\beta\cdot x_j$ to obtain 
\[
\iota^*_{F_\mathbf{0}}(\widetilde I_{Z_A}) = \sum_\beta \widetilde C_\beta \cdot J_\beta\cdot (q')^\beta \sum _{a\geq 0} N^{PAX}_{\beta,a} \cdot  (q_1')^{a+\sum_{i=1}^r \beta\cdot x_i},
\]

And we notice that 
\begin{equation}\label{eq:q1PAX}
		q_1^{a-\sum_{j=1}^s \beta\cdot x_{j+r}} =q_1^{a+\sum_{i=1}^r \beta\cdot x_{i}}\cdot q_1^{-\sum_{i=1}^m \beta\cdot x_i}
	\end{equation}

After applying Propostion~\ref{p.pax/paxy}, the remainder of the details are exactly as in the proof of Theorem~\ref{Them.building_block}
\end{proof}

\bibliographystyle{alpha}
\bibliography{references}

\newcommand{\etalchar}[1]{$^{#1}$}
\begin{thebibliography}{CFFG{\etalchar{+}}23}

\bibitem[{\'A}CGP03]{AG03}
Luis {\'A}lvarez-C{\'o}nsul and Oscar Garc{\'\i}a-Prada.
\newblock Hitchin--{K}obayashi correspondence, quivers, and vortices.
\newblock {\em Communications in mathematical physics}, 238:1--33, 2003.

\bibitem[Ber09]{Ber}
Marie-Am{\'e}lie Bertin.
\newblock Examples of {C}alabi--{Y}au 3-folds of $\mathbb{P}^7$ with $\rho= 1$.
\newblock {\em Canadian Journal of Mathematics}, 61(5):1050--1072, 2009.

\bibitem[BPZ15]{BPZ}
Francesco Benini, Daniel~S. Park, and Peng Zhao.
\newblock Cluster algebras from dualities of 2d {$N=(2,2)$} quiver gauge
  theories.
\newblock {\em Comm. Math. Phys.}, 340(1):47--104, 2015.

\bibitem[Bro14]{Brown}
Jeff Brown.
\newblock Gromov--{W}itten invariants of toric fibrations.
\newblock {\em International Mathematics Research Notices},
  2014(19):5437--5482, 2014.

\bibitem[CFFG{\etalchar{+}}23]{CFGKS}
Ionut Ciocan-Fontanine, David Favero, J\'{e}r\'{e}my Gu\'{e}r\'{e}, Bumsig Kim,
  and Mark Shoemaker.
\newblock Fundamental factorization of a {GLSM} {P}art {I}: {C}onstruction.
\newblock {\em Mem. Amer. Math. Soc.}, 289(1435):iv+96, 2023.

\bibitem[CLS22]{CLS}
Tom Coates, Wendelin Lutz, and Qaasim Shafi.
\newblock The abelian/nonabelian correspondence and {G}romov–{W}itten
  invariants of blow-ups.
\newblock {\em Forum of Mathematics, Sigma}, 10:e67, 2022.

\bibitem[Coa]{CoatesQLHT}
Tom Coates.
\newblock The quantum {L}efschetz principle for vector bundles as a map between
  {G}ivental cones.
\newblock {\em arxiv:1405.2893}.

\bibitem[CZ23]{CZ}
Yalong Cao and Gufang Zhao.
\newblock Quasimaps to quivers with potentials.
\newblock arxiv:2306.01302, 2023.

\bibitem[Don20]{Don}
Hai Dong.
\newblock {\em $I$-funciton in Grassmannian Duality}.
\newblock PhD thesis, Peking Univeristy, 2020.

\bibitem[DW22]{DW}
Hai Dong and Yaoxiong Wen.
\newblock Level correspondence of the k-theoretic i-function in grassmann
  duality.
\newblock {\em Forum of Mathematics, Sigma}, 10:e44, 2022.

\bibitem[FJR18]{FJR2}
Huijun Fan, Tyler Jarvis, and Yongbin Ruan.
\newblock A mathematical theory of the gauged linear sigma model.
\newblock {\em Geom. Topol.}, 22(1):235--303, 2018.

\bibitem[FK20]{FKim}
David Favero and Bumsig Kim.
\newblock General glsm invariants and their cohomological field theories.
\newblock unpublished, 2020.

\bibitem[Giv98]{Giv}
Alexander~B Givental.
\newblock Elliptic {G}romov-{W}itten invariants and the generalized mirror
  conjecture.
\newblock {\em arXiv preprint math/9803053}, 1998.

\bibitem[GN71]{GN}
Tor~H Gulliksen and OG~Neg{\aa}rd.
\newblock Un complexe resolvant pour certain id{\'e}aux
  d{\'e}t{\`e}rminentiels.
\newblock {\em Preprint series: Pure mathematics http://urn. nb. no/URN: NBN:
  no-8076}, 1971.

\bibitem[GP01]{GP1}
Mark Gross and Sorin Popescu.
\newblock Calabi--{Y}au threefolds and moduli of abelian surfaces {I}.
\newblock {\em Compositio Mathematica}, 127(2):169--228, 2001.

\bibitem[GP11]{GP2}
Mark Gross and Sorin Popescu.
\newblock Calabi-{Y}au three-folds and moduli of abelian surfaces {II}.
\newblock {\em Transactions of the American Mathematical Society},
  363(7):3573--3599, 2011.

\bibitem[Har13]{Har}
Joe Harris.
\newblock {\em Algebraic geometry: a first course}, volume 133.
\newblock Springer Science \& Business Media, 2013.

\bibitem[HM73]{HM}
Geoffrey Horrocks and David Mumford.
\newblock A rank 2 vector bundle on p4 with 15,000 symmetries.
\newblock {\em Topology}, 12(63-81):6, 1973.

\bibitem[HM18]{HM18}
Yoshinori Honma and Masahide Manabe.
\newblock Determinantal {C}alabi-{Y}au varieties in {G}rassmannians and the
  {G}ivental {I}-functions.
\newblock {\em Journal of High Energy Physics}, 2018(12):1--50, 2018.

\bibitem[HT13]{HT}
Shinobu Hosono and Hiromichi Takagi.
\newblock Mirror symmetry and projective geometry of reye congruences i.
\newblock {\em Journal Algebraic Geometry}, 2013.

\bibitem[HZ23]{HZ}
Weiqiang He and Yingchun Zhang.
\newblock Seiberg duality conjecture for star-shaped quivers and finiteness of
  {G}romov-{W}itten thoery for {D}-type quivers.
\newblock {\em arXiv preprint arXiv:2302.02402}, 2023.

\bibitem[JKL{\etalchar{+}}12]{JKL}
Hans Jockers, Vijay Kumar, Joshua~M. Lapan, David~R. Morrison, and Mauricio
  Romo.
\newblock Nonabelian 2{D} gauge theories for determinantal {C}alabi-{Y}au
  varieties.
\newblock {\em J. High Energy Phys.}, (11):166, front matter + 46, 2012.

\bibitem[Kim99]{Bum}
Bumsig Kim.
\newblock Quantum hyperplane section theorem for homogeneous spaces.
\newblock {\em Acta Math}, 183:71–99, 1999.

\bibitem[KK10]{KK}
Micha{\l} Kapustka and Grzegorz Kapustka.
\newblock A cascade of determinantal {C}alabi--{Y}au threefolds.
\newblock {\em Mathematische Nachrichten}, 283(12):1795--1809, 2010.

\bibitem[Las78]{Las}
Alain Lascoux.
\newblock Syzygies des variétés déterminantales,.
\newblock {\em Advances in Mathematics}, 30(3):202--237, 1978.

\bibitem[Lee01]{Lee}
Y-P Lee.
\newblock Quantum {L}efschetz hyperplane theorem.
\newblock {\em Inventiones mathematicae}, 145(1):121--149, 2001.

\bibitem[LLY99a]{LLY1}
Bong~H Lian, Kefeng Liu, and Shing-Tung Yau.
\newblock Mirror principle {I}.
\newblock {\em Surveys in Differential Geometry}, 5(1):405--454, 1999.

\bibitem[LLY99b]{LLY2}
Bong~H. Lian, Kefeng Liu, and Shing-Tung Yau.
\newblock Mirror principle. {II}.
\newblock volume~3, pages 109--146. 1999.
\newblock Sir Michael Atiyah: a great mathematician of the twentieth century.

\bibitem[LLY99c]{LLY3}
Bong~H. Lian, Kefeng Liu, and Shing-Tung Yau.
\newblock Mirror principle. {III}.
\newblock {\em Asian J. Math.}, 3(4):771--800, 1999.

\bibitem[LR24]{LR}
Ban {Lin} and Mauricio {Romo}.
\newblock {B-brane Transport and Grade Restriction Rule for Determinantal
  Varieties}.
\newblock {\em arXiv e-prints}, page arXiv:2402.07109, February 2024.

\bibitem[Oh21]{Oh}
Jeongseok Oh.
\newblock Quasimaps to {GIT} fibre bundles and applications.
\newblock {\em Forum of Mathematics, Sigma}, 9:e56, 2021.

\bibitem[Ott95]{Otta}
Giorgio Ottaviani.
\newblock {\em Variet\'a proiettive di codimensione piccola}.
\newblock Ist. Nazion. Di alta matematica f. Severi. Aracne, 1995.

\bibitem[Rua17]{Ruan17}
Yongbin Ruan.
\newblock Nonabelian gauged linear sigma model.
\newblock {\em Chinese Annals of Mathematics, Series B}, 38(4):963--984, 2017.

\bibitem[Sch86]{Sch}
Chad Schoen.
\newblock On the geometry of a special determinantal hypersurface associated to
  the mumford-horrocks vector bundle.
\newblock {\em Journal f{\"u}r die reine und angewandte Mathematik (Crelle's
  Journal)}, 1986:111 -- 85, 1986.

\bibitem[Sho]{Shoe}
Mark Shoemaker.
\newblock Towards a mirror theorem for {GLSM}s.
\newblock https://arxiv.org/abs/2108.12360.

\bibitem[ST22]{SupTse}
Shidhesh Supekar and Hsian-Hua Tseng.
\newblock Wall-crossing for quasimaps to {GIT} stack bundles.
\newblock arXiv:arXiv:2212.00910, December 2022.

\bibitem[TX17]{TX2}
Gang Tian and Guangbo Xu.
\newblock The symplectic approach of gauged linear {$\sigma$}-model.
\newblock In {\em Proceedings of the {G}\"{o}kova {G}eometry-{T}opology
  {C}onference 2016}, pages 86--111. G\"{o}kova Geometry/Topology Conference
  (GGT), G\"{o}kova, 2017.

\bibitem[Zha22]{Zha}
Yingchun Zhang.
\newblock {G}romov-{W}itten {T}heory of ${A}_n$ type quiver varieties and
  {S}eiberg {D}uality.
\newblock arXiv:2112.11812v3, July 2022.

\end{thebibliography}

\end{document}